\documentclass[11pt]{amsart}
\makeatletter
\@namedef{subjclassname@2020}{\textup{2020} Mathematics Subject Classification}
\makeatother

\usepackage{enumerate,enumitem,tikz-cd,mathtools,amssymb}

\usepackage[a4paper,top=2.5cm,bottom=2.5cm,left=2.5cm,right=2.5cm,marginparwidth=1.75cm]{geometry}

\usepackage{verbatim}

\definecolor{cite}{RGB}{44,123,182}
\definecolor{ref}{RGB}{215,25,28}

\usepackage[pdftex,
              pdfauthor={Marin Petkovic, Franco Rota},
              pdftitle={A note on the Kuznetsov component of the Veronese double cone},
              pdfsubject={Algebraic geometry, stability condition, Fano threefolds},
              pdfkeywords={stability moduli Fano  petkovic rota},
              colorlinks=true,
              citecolor=cite,
              linkcolor=ref
              ]{hyperref}

\usepackage[all,cmtip]{xy}

\usepackage{libertine}

\makeatletter
\newtheorem*{rep@theorem}{\rep@title}
\newcommand{\newreptheorem}[2]{%
\newenvironment{rep#1}[1]{%
 \def\rep@title{#2 \ref{##1}}%
 \begin{rep@theorem}}%
 {\end{rep@theorem}}}
\makeatother

\theoremstyle{plain}
\newtheorem{thm}{Theorem}[section]
\newtheorem{corollary}[thm]{Corollary}
\newtheorem{lemma}[thm]{Lemma}

\newtheorem{proposition}[thm]{Proposition}

\newtheorem*{thm*}{Theorem}
\newtheorem*{corollary*}{Corollary}
\newtheorem*{lemma*}{Lemma}
\newtheorem*{ld*}{Lemma/Definition}
\newtheorem*{proposition*}{Proposition}
\newtheorem*{assumption*}{Assumption}

\theoremstyle{definition}
\newtheorem{definition}[thm]{Definition}
\newtheorem{remark}[thm]{Remark}

\newtheorem*{definition*}{Definition}
\newtheorem*{example*}{Example}
\newtheorem*{xca*}{Exercise}
\newtheorem*{claim*}{Claim}
\newtheorem*{fact*}{Fact}
\newtheorem*{notation*}{Notation}
\newtheorem*{construction*}{Construction}
\newtheorem*{ack*}{Acknowledgements}
\newtheorem*{question*}{Question}
\newtheorem*{problem*}{Problem}
\newtheorem*{conjecture*}{Conjecture}

\newcommand{\C}{\mathbb{C}}
\newcommand{\Z}{\mathbb{Z}}
\newcommand{\Q}{\mathbb{Q}}
\newcommand{\R}{\mathbb{R}}

\newcommand{\pr}[1]{\mathbb P^{#1}}

\DeclareMathOperator{\GL}{GL}

\DeclareMathOperator{\coker}{coker}
\DeclareMathOperator{\Hom}{Hom}
\DeclareMathOperator{\Ext}{Ext}

\DeclareMathOperator{\Sym}{Sym}
\DeclareMathOperator{\Ch}{ch}

\DeclareMathOperator{\Pic}{Pic\,}
\DeclareMathOperator{\Hilb}{Hilb}

\DeclareMathOperator{\Stab}{Stab}

\DeclareMathOperator{\Coh}{Coh}

\newcommand{\sA}{\mathcal{A}}
\newcommand{\sB}{\mathcal{B}}
\newcommand{\sH}{\mathcal{H}}
\newcommand{\sO}{\mathcal{O}}
\newcommand{\sK}{\mathcal{K}}

\newcommand{\sQ}{\mathcal{Q}}

\newcommand{\sM}{\mathcal{M}}

\newcommand{\PP}{\mathcal{P}}

\newcommand{\st}{\,|\,}					

\usepackage{leftidx}						

\usepackage{xcolor}						

\DeclarePairedDelimiter{\set}{\lbrace}{\rbrace}

\DeclarePairedDelimiter{\pair}{\langle}{\rangle}

\DeclarePairedDelimiter{\abs}{\lvert}{\rvert}

\newcommand{\Ku}{\mathsf{Ku}}

\newcommand{\sI}{\mathcal{I}}

\begin{document}

\title{A note on the Kuznetsov component of the Veronese double cone}

\author[M. Petkovi\'{c}]{Marin Petkovi\'{c}}
\address{MP: Department of Mathematics \\ University of Utah \\ Salt Lake City,~
UT 84102,~USA} 
\email{petkovic@math.utah.edu}

\author[F. Rota]{Franco Rota}
\address{FR: School of Mathematics and Statistics \\ University of Glasgow \\ University Place, Glasgow,~
G128QQ,~UK} 
\email{franco.rota@glasgow.ac.uk}


\begin{abstract}

This note describes moduli spaces of complexes in the derived category of a Veronese double cone $Y$. Focusing on objects with the same class $\kappa_1$ as ideal sheaves of lines, we describe the moduli space of Gieseker stable sheaves and show that it has two components. Then, we study the moduli space of stable complexes in the Kuznetsov component of $Y$ of the same class, which also has two components. One parametrizes ideal sheaves of lines and it appears in both moduli spaces. The other components are not directly related by a wall-crossing: we show this by describing an intermediate moduli space of complexes as a space of stable pairs in the sense of Pandharipande and Thomas. 
\end{abstract}

\maketitle

\section{Introduction}

Let $Y$ be a Fano threefold of Picard rank 1 and index 2, that is $\Pic(Y)= \pair H $ with $-K_Y\sim 2H$ ample. The manifold $Y$
belongs to one of five families of deformations, indexed by their degree $d\coloneqq H^3\in \set{1,...,5}$ \cite{Iskovskih1977}.
For all values of $d$, $D^b(\Coh Y)$ admits a triangulated subcategory $\Ku(Y)$ - called the \textit{Kuznetsov component of $Y$} \cite{Kuz09_threefolds} - which is the right orthogonal to an exceptional pair of line bundles.

The numerical Grothendieck group $ N(\Ku(Y))\subset N(D^b(Y))$ is a rank 2 lattice generated by the classes \footnote{We will use $H$ to also indicate the class $[H]$ in the numerical Grothendieck group and the Chern character $\Ch(H)$ in the cohomology ring.} 
$$\kappa_1=1-\frac{H^2}{d}\qquad\mbox{ and }\qquad \kappa_2=H-\frac{H^2}{2}-\frac{(6-d)H^3}{6d}.$$ 

Here we work with $Y$ general of degree $1$. In this case, $Y$ is a hypersurface of degree 6 in the weighted projective space $\mathbb P(1,1,1,2,3)$ and is called a \textit{Veronese double cone}. This paper studies moduli spaces of objects of class $\kappa_1$: these are closely related to the geometry of lines in $Y$, since ideal sheaves of lines in $Y$ have class $\kappa_1$.

\smallskip

Our first result is a description of the Hilbert scheme of lines in $Y$, denoted $\Hilb(Y,t+1)$.

\begin{thm}[= (\ref{cor_Hilb=M_G}) + (\ref{thm_Gieseker})]
$\Hilb(Y,t+1)$ is isomorphic to the moduli space $M_G(\kappa_1)$ of Gieseker-stable sheaves of class $\kappa_1$. 

It has two irreducible components $M_1$ and $M_2$.
$M_1$ is a smooth surface compactifying the locus of ideals of lines of $Y$. 
$M_2$ has dimension 5 and its general points parametrize genus 1 curves union a point. It is smooth outside the intersection with $M_1$. 
Points in $M_1\cap M_2$ parametrize singular rational curves with an embedded point at the singularity. 
\end{thm}

The first very ample multiple of $H$ is $3H$, so $M_2$ is analogous to the extra component in the Hilbert scheme of twisted cubics in projective space \cite{PS85}. 

\smallskip

Two irreducible components also appear in moduli spaces of complexes in $\Ku(Y)$. A general construction of Bayer, Lahoz, Macr\'i, and Stellari yields a stability condition $\sigma$ on $\Ku(Y)$ \cite{BLMS17}. Therefore, moduli spaces of $\sigma$-semistable complexes in $\Ku(Y)$ of class $v$ are defined, and denoted $M_{\sigma}(v)$.  
Combined with the rotation autoequivalence introduced in \cite[Sec. 3.3]{Kuz15_CY}, a construction of \cite{PY20}
induces isomorphisms of moduli spaces of $\sigma$-stable complexes in $\Ku(Y)$ of different classes. When applied to the Veronese double  cone, this isomorphism identifies $M_\sigma(\kappa_1)$ with $M_\sigma(\kappa_2)$, which has been studied in \cite{APR19} and has a component $M_3$ isomorphic to $Y$ itself.

\begin{thm}[= (\ref{thm_three_moduli})]
Let $Y$ be a general smooth Veronese double cone. The moduli spaces $M_\sigma(\kappa_1)$, $M_\sigma(\kappa_2)$ and $M_\sigma(\kappa_2-\kappa_1)$ are isomorphic. 
They have two irreducible components $M_1$ and $M_3$, isomorphic respectively to the 2-dimensional component of $M_G(\kappa_1)$ and to $Y$ itself. 
\end{thm}

The objects parametrized by $M_3$ are described explicitly in Section \ref{sec_obj_and_mod_in_ku} and they are related to projections of skyscraper sheaves of points to $\Ku(Y)$.

\smallskip

The moduli spaces $M_G(\kappa_1)$ and $M_{\sigma}(\kappa_1)$ are related by deformations of (weak) stability conditions and wall-crossing. The interpolating stability conditions $\sigma_{\alpha,\beta}$ and $\sigma_{\alpha,\beta}^0$
yield moduli spaces denoted respectively by $M_{\alpha,\beta}^0(\kappa_1)$ and $M_{\alpha,\beta}(\kappa_1)$. 
Theorem \ref{thm_classification} describes them set-theoretically, by classifying complexes that are stable for $\sigma_{\alpha,\beta}$ and $\sigma_{\alpha,\beta}^0$. As a consequence, $\sigma_{\alpha,\beta}^0$-semistable complexes can be interpreted as quotients of $\sO_Y$ of class $\kappa_1$, in a perverse (repeatedly tilted) heart on $Y$ (Prop. \ref{prop_quot}).
We also obtain the following moduli-theoretic description, which relates $M_{\alpha,\beta}^0(\kappa_1)$ and the moduli space $P(\kappa_1)$ of Pandharipande-Thomas stable pairs of class $\kappa_1$ \cite{PT09_CurveCounting}. 

\begin{thm}[= (\ref{thm_classification}) + (\ref{thm_stable_pairs})]
The space $M_{\alpha,\beta}(\kappa_1)$ coincides with $M_G(\kappa_1)$.

The space $M_{\alpha,\beta}^0(\kappa_1)$ is identified with $P(\kappa_1)$, so it is a projective scheme. It contains $M_1$ and a second irreducible component $\tilde{M}_3$, which is the blow-up of $M_3\simeq Y$ at a point.
\end{thm}

\smallskip

In summary, the spaces 
\[\Hilb(Y,t+1)\simeq M_G(\kappa_1)\simeq M_{\alpha,\beta}(\kappa_1)\]
have two irreducible components $M_1$ and $M_2$, whose generic points parametrize lines and genus 1 curves union a point respectively. In the space $M^0_{\alpha,\beta}\simeq P(\kappa_1)$ the component $M_2$ is traded off with $\tilde{M_3}$, which is related to projections of points to $\Ku(Y)$ and is a blow-up of $M_3\simeq Y\subset M_\sigma(\kappa_1) = M_1 \cup M_3$.

\subsection*{Related works and remarks}

For degrees $d>1$, \cite{PY20} shows that the spaces $\Hilb(Y,t+1)$ and $M_\sigma(\kappa_1)$ are isomorphic. They both coincide with an irreducible surface (smooth for $d\geq 3$), called the Fano surface of lines of $Y$\footnote{The statements about the Hilbert scheme are classical, see for example \cite[\S 2.2]{Kuznetsov2018} and references therein. The isomorphism between $\Hilb(Y,t+1)$ and $M_\sigma(\kappa_1)$ is \cite[Theor. 1.1]{PY20}. If $d=1$, the closure of the locus of smooth lines in $\Hilb(Y,t+1)$ is a projective irreducible scheme, given by a smooth surface with an embedded curve \cite[Theor. 4]{Tihomirov1982}.}. Thus, the appearance of a second component is special to degree 1. As mentioned above, this is linked to neither $H$ nor $2H$ being very ample for Veronese double cones. 

The unusual behavior of $M^0_{\alpha,\beta}$ also appears only in degree 1. 
In fact, the authors of \cite{PY20} show that the moduli spaces $M_{\alpha,\beta}^0(\kappa_1)$ and $M_{\alpha,\beta}(\kappa_1)$ are all isomorphic to $M_\sigma(\kappa_1)$ for $d>1$. The same happens for moduli spaces of class $\kappa_2$ - studied for all degrees in \cite{APR19} - and in the context of cubic fourfolds as well \cite{BLMS17}. 

It is worth remarking that, in the cases of $d>1$ and of cubic fourfolds, the heart $\sA(\alpha,\beta)$ of the stability condition $\sigma$ has dimension $\leq 2$, which is linked to results of smoothness of moduli spaces (e.g. in \cite[Theor. 1.2]{PY20}) and is crucial to prove \textit{categorical Torelli theorems}\footnote{These are reconstruction results showing that $\Ku(Y)$ determines $Y$ up to isomorphism. Categorical Torelli theorems are known to hold for all degrees $d >1$, but the question is open for $d=1$. We direct the interested reader to \cite{PS22}, which surveys results and open problems in this area.}. As shown in Remark \ref{rmk_hom_dim}, however, $\sA(\alpha,\beta)$ has dimension 3 in degree 1.  

More recently, the works \cite{Qin21} and \cite{LZ21} study moduli spaces of non-primitive classes and relate them to instanton bundles on $Y$.

\subsection*{Structure of the paper} After introducing preliminary notions in Section \ref{sec_preliminaries}, we study the Hilbert scheme of lines on a Veronese double cone in Section \ref{sec_Fano_scheme_of_lines}. Section \ref{sec_obj_and_mod_in_ku} is dedicated to the description of the moduli space $M_\sigma(\kappa_1)$. Section \ref{sec_set_theoretic} contains the classification of semistable objects for the interpolating weak stability conditions, and Section \ref{sec_stable_pairs} contains the description of the moduli space of stable pairs.

\subsection*{Acknowledgements} We are grateful to Laura Pertusi, Song Yang, Aaron Bertram and Arend Bayer for the fruitful discussions on these topics. Both authors were partially supported by NSF-FRG grant DMS 1663813, PI: Aaron Bertram.

\section{Preliminaries}\label{sec_preliminaries}

\subsection{Stability conditions}
\label{sec_stab_cond_and_Ku}

Here we give a short review of Bridgeland stability conditions, with the main purpose of fixing the notation for what follows. We direct the interested reader to the seminal work of Bridgeland \cite{Bri07_triang_cat} and to the survey \cite{MS17} and references therein for a thorough description.

\begin{definition}
Let $\sA$ be an abelian category and let $K(\sA)$ be its Grothendieck group. A \emph{(weak) stability function} is a group homomorphism $Z:K(\sA)\to \C$ such that
$$\Im Z(E)> 0 \text{ or }\Im Z(E)=0 \text{ and } \Re Z(E)<(\leq) 0$$
for any $0\neq E\in \sA$. To a (weak) stability function $Z$, we associate a \emph{slope function}
$$\mu(E)=\begin{cases}
\frac{-\Re Z(E)}{\Im Z(E)} & \text{if }\Im Z(E)\neq 0\\
+\infty & \text{otherwise}
\end{cases}$$
We say that $E\in\sA$ is \emph{stable} if for all quotients $E \twoheadrightarrow F$ in $\sA$ we have $$\mu(E)<\mu(F).$$ 
Similarly, $E$ is said to be \textit{semistable} if only the non-strict inequality $\mu(E)\leq \mu(F)$ holds.
\end{definition}
\begin{definition}
Let $\mathsf T$ be  a triangulated category and $v:K(\mathsf T)\twoheadrightarrow \Lambda$ a surjection from the Grothendieck group of $\mathsf T$ to a finite rank lattice. A \emph{(weak) stability condition}  on a triangulated category $\mathsf T$ (with respect to $v$) is a pair $\sigma=(\sA,Z)$ consisting of
\begin{itemize}
    \item[--] a heart of a bounded $t$-structure $\sA$
    \item[--] a (weak) stability function $K(\sA)\xrightarrow[]{v} \Lambda\xrightarrow[]{Z} \C$
\end{itemize}
satisfying the following properties:
\begin{enumerate}[label=(\roman*)]
        \item (Harder-Narasimhan filtration) Any $E\in \sA$ has a filtration in $\sA$ with semistable quotients with decreasing slopes.
        \item (Support property) There exists a quadratic form $Q$ on $\Lambda\otimes \R$ which is negative definite on $\ker Z$ and for all semistable $E\in \sA$ we have $Q(E)\geq0 $ .
    \end{enumerate}
    We say an object $E\in \mathsf T$ is $\sigma$-(semi)stable if $E[k]\in \sA$ for some $k\in \Z$ and $E[k]$ is semistable with respect to $Z$ \footnote{\label{foot_Slicing}Given a stability condition $\sigma=(\sA,Z)$ and $\phi\in(0,1]$, let $\mathcal{P}(\phi)$ be the category of $\sigma$-semistable objects  $E\in \sA$ satisfying $Z(v(E))\in \R_{>0}e^{i\pi\phi}$. Then, define $\mathcal P(\phi)$ for all $\phi\in \R$ by imposing $\mathcal{P}(\phi+1)\coloneqq \mathcal{P}(\phi)[1]$.
    The collection $\mathcal P\coloneqq \{\mathcal P(\phi)\}_{\phi\in \R}$ defines a \textit{slicing} of $\mathsf T$. The datum of a slicing and a compatible stability function is in fact equivalent to that of a stability condition \cite[Prop. 5.3]{Bri07_triang_cat}, and it is sometimes convenient to identify $\sigma$ with the pair $(Z,\PP)$.}.
\end{definition}

\begin{definition}
Let $\sigma=(\sA,Z)$ be  a weak stability condition on $\mathsf T$. For $\beta\in \R$, we define subcategories
$\sA_{\mu\leq \beta}$ and $\sA_{\mu>\beta}$ consisting of objects $E$ such that slopes of all Harder-Narasimhan factors of $E$ are $\leq \beta$ and $>\beta$ respectively.
The tilt of $\sA$ is then defined as the extension closure of $\sA_{\mu\leq \beta}[1]$ and $\sA_{\mu>\beta}$ and denoted
$$\sA_\sigma^\beta=\Big[\sA_{\mu_\sigma\leq \beta}[1],\sA_{\mu_\sigma>\beta}\Big].$$
That is, objects $E\in \sA_\sigma^\beta$ are complexes with 
\begin{align*}
    &\sH^{-1}_\sA(E)\in \sA_{\mu_\sigma\leq \beta}\\
    &\sH^0_\sA(E)\in \sA_{\mu_\sigma>\beta}\\
    &\sH^i_\sA(E)=0,\text{ for }i\neq -1,0.
\end{align*}
\end{definition}

For a smooth projective variety $Y$ with a hyperplane class $H$, define the map
\begin{equation*}
    v=(H^3\Ch_0,H^2\Ch_1,H\Ch_2):K(Y)\to \Q^3
\end{equation*} 
and let $\Lambda\simeq \Z^{\oplus 3}$ be its image.
In this paper, we will be working with the following weak stability conditions on $D^b(Y)$:
\subsubsection{Slope stability}\label{def_Mumford-stability} $\sigma_M=(\Coh(Y),-H^2\Ch_1+iH^3\Ch_0)$ is a weak stability condition with respect to the rank 2 lattice defined as the image of $(H^3\Ch_0,H^2\Ch_1):K(Y)\to \Z^2$.
    This stability condition is also called Mumford stability, or slope stability. We will denote the corresponding slope function with $\mu_M$.
\subsubsection{Tilt-stability}\label{def_tilt-stability} $\sigma_{\alpha,\beta}=(\Coh^\beta(Y),Z_{\alpha,\beta})$, for $\alpha>0$ and $\beta\in \R$, where 
$$\Coh^\beta(Y)=\Big[\Coh(Y)_{\mu_M\leq \beta}[1],\Coh(Y)_{\mu_M>\beta}\Big]$$
and
$$Z_{\alpha,\beta}(E)=-H\Ch_2^\beta (E)+\frac{\alpha^2}{2} H^3\Ch_0 (E)+ i \left(H^2\Ch_1(E)-\beta H^3\Ch_0 (E)\right).$$
Here, $\Ch^\beta(-)\coloneqq e^{-\beta H}\cdot \Ch(-)$ is the twisted Chern character.
This is a weak stability condition with respect to the lattice $\Lambda$ above \cite{BMT14, BMS16}, and is usually called tilt-stability. The corresponding slope function will be denoted with $\mu_{\alpha,\beta}$. The quadratic form satisfying the support property is \cite[Cor. 7.3.2]{BMT14}:
\[ Q(E) = (H^2 \Ch_1^\beta(E)) - 2(H\Ch_2^\beta(E))(H^3 \Ch_0(E)). \]
    
For a class $w\in \Lambda$, the half-plane $\{ (\alpha,\beta) \mid \alpha >0, \beta\in \R \}$ admits a wall-and-chamber decomposition:

\begin{definition}
\label{def_of_walls}
A \textit{numerical wall} with respect to $w\in \Lambda$ is the solution set in $\{ (\alpha,\beta) \mid \alpha >0, \beta\in \R \}$ of an equation $\mu_{\alpha,\beta}(w)=\mu_{\alpha,\beta}(u)$ for some $u\in \Lambda$.

A subset of a numerical wall for $w$ is an \textit{actual wall} if there exists a short exact sequence of semistable complexes in $\Coh^\beta(Y)$, $0\to F\to E\to G \to 0$, with $v(E)=w$ and $v(F)$ defining the numerical wall. 
\end{definition}

Walls of tilt-stability satisfy Bertram's Nested Wall Theorem (first proven for surfaces in \cite{Maciocia14}). In particular:

\begin{thm}[{\cite[Theor. 3.3]{Sch20}}]\label{thm_Nested_Wall}
Fix $w\in \Lambda$.
\begin{itemize}
    \item[--] numerical walls are nested semicircles centered on the $\beta$-axis, except for possibly one, which is a half-line with constant $\beta$;
    \item[--] if two numerical walls intersect, then they coincide;
    \item[--] if a point of a numerical wall is an actual wall, then the whole numerical wall is an actual wall. 
\end{itemize}
\end{thm}

We then define \textit{chambers} as connected components of complements of actual walls. If $(\alpha,\beta)$ and $(\alpha',\beta')$ belong to the same chamber, then an object $E$ of class $w$ is $\sigma_{\alpha,\beta}$-semistable if and only if it is $\sigma_{\alpha',\beta'}$-semistable.

\subsubsection{Rotation of tilt-stability}\label{def_tilt0-stability} $\sigma^0_{\alpha,\beta}=(\Coh^0_{\alpha,\beta}(Y),Z^0_{\alpha,\beta})$, for $\alpha>0$ and $\beta\in \R$, where
    $$\Coh^0_{\alpha,\beta}(Y)=\Big[\Coh^\beta(Y)_{\mu_{\alpha,\beta}\leq 0}[1],\Coh^\beta(Y)_{\mu_{\alpha,\beta}>0}\Big]$$
    and
    \begin{equation}\label{eq_def_of_Z^0}
      Z^0_{\alpha,\beta}(E)=-iZ_{\alpha,\beta}(E) 
    \end{equation}
    This is also a weak stability condition with respect to $\Lambda$ (\cite[Prop 2.15]{BLMS17}). The corresponding slope function will be denoted with $\mu^0_{\alpha,\beta}$.

Like for tilt-stability, one can define walls and chambers for $\sigma_{\alpha,\beta}^0$ by replacing $\mu_{\alpha,\beta}$ with $\mu_{\alpha,\beta}^0$ and $\Coh^\beta(Y)$ with $\Coh_{\alpha,\beta}^0(Y)$ in Definition \ref{def_of_walls}.

\subsection{Kuznetsov component}
\label{sec_KuzComp}

Let $Y$ be a smooth Fano threefold of Picard rank $1$ and index $2$. 
The derived category of $Y$ admits a semi-orthogonal decomposition
\[ D^b(Y)=\pair{\Ku(Y),\sO_Y,\sO_Y(1)} \]
where the admissible subcategory $\Ku(Y)$ is called the Kuznetsov component \cite{Kuz09_threefolds}. The numerical Grothendieck group $ N(\Ku(Y))\subset N(D^b(Y))$ has rank 2 and is  generated by the classes $$\kappa_1=[ I_\ell]=1-\frac{H^2}{d}\qquad\&\qquad \kappa_2=H-\frac{H^2}{2}-\frac{(6-d)H^3}{6d}.$$
In this basis, the Euler form writes
\[ \begin{pmatrix} -1 & -1 \\ 1-d & -d \end{pmatrix}. \]
It is negative definite, and if $d=1$ the only $-1$ classes are $\pm \kappa_1, \pm \kappa_2$ and $\pm (\kappa_1-\kappa_2)$.

Recall that for $E\in D^b(Y)$ exceptional, the left mutation $\mathbb L_E(-)$ across $E$ is the functor sending $G \in D^b(Y)$ to the cone of the evaluation map $ev$:
\[ \mathbf R\Hom(E,G)\otimes E \xrightarrow{ev} G \to \mathbb L_E(G). \]
The inclusion $\Ku(Y)\subset D^b(Y)$ has an adjoint projection functor $\pi\coloneqq \mathbb L_{\sO_Y} \circ \mathbb L_{\sO_Y(1)}$.

The category $\Ku(Y)$ admits an autoequivalence called the \textit{rotation functor} $$\mathsf R(-)\coloneqq \mathbb L_{\sO_Y} ( - \otimes \sO_Y(1)),$$ and a Serre functor. In fact, the two are related:

\begin{lemma}
\label{lem_serre_dual}
The Serre functor on $\Ku(Y)$ satisfies
\[ \mathsf S_{\Ku(Y)}^{-1} \simeq \mathsf R^2[-3]. \]
\end{lemma}
\begin{proof}
By \cite[Lemma 2.7]{Kuz14}, we have that $\mathsf S_{\Ku(Y)}^{-1}\simeq \pi \circ \mathsf S_Y^{-1}$.
It is then straightforward to check that 
\[ \pi \mathsf S_Y^{-1}(E)=\pi(E(2))[-3] = \mathbb L_\sO (\mathbb L_{\sO(1)} (E(2)))[-3] \simeq \mathsf R^2(E)[-3]. \qedhere\]
\end{proof}

One of the results of \cite{BLMS17} is that $\Ku(Y)$ supports stability conditions. Define the set
\begin{equation}
\label{eq_def_of_V}
    V=\Big\{(\alpha,\beta)\in \R_{>0}\times \R\st 0<\alpha<\min\{-\beta,\beta+1\},-1<\beta<0\Big\},
\end{equation}
then we have:
\begin{thm}[{\cite[Theor. 6.8]{BLMS17}}]
\label{thm_stab_cond_on_Ku}
For any $(\alpha,\beta)\in V$, the weak stability condition  $\sigma^0_{\alpha,\beta}$ from Section  \ref{def_tilt0-stability}, induces a Bridgeland stability condition $\sigma(\alpha,\beta)$ on $\Ku(Y)$, with heart given by
$$\sA(\alpha,\beta)\coloneqq \Coh^0_{\alpha,\beta}(Y)\cap \Ku(Y)$$
and central charge $Z_{\alpha,\beta|K(\sA)}^0$.
We will denote the slope function of $\sigma(\alpha,\beta)$ with $\mu(\alpha,\beta)$.
\end{thm}

The set of stability conditions on $\Ku(Y)$ is denoted $\Stab(\Ku(Y))$, it is a complex manifold and it admits the following group actions:
\begin{itemize}
    \item[--] The universal cover $\widetilde \GL_2^+(\R)$ acts on the right: an element of $\widetilde \GL_2^+(\R)$ is a pair $\tilde g=(g,M)$ where $g\colon \R \to \R$ is an increasing function such that $g(\phi+1)=g(\phi)+1$, and $M\in \GL_2^+(\R)$ satisfies $Me^{i\phi\pi} \in \R_{>0}e^{ig(\phi)\pi}$. Given a stability condition $\sigma= (Z,\PP)\in \Stab(\Ku(Y))$, we define $\sigma \cdot \tilde g = (Z',\PP')$ to be the stability condition with $Z'=M^{-1}\circ Z$ and $\PP'(\phi)=\PP(g(\phi))$ (see footnote \ref{foot_Slicing}). Stability is preserved under this action: an object $E\in \Ku(Y)$ is $\sigma$-stable if and only if it is $\sigma\cdot \tilde g$-stable for all $\tilde g\in \widetilde \GL_2^+(\R)$. 
    \item[--] An autoequivalence $\Phi$ of $\mathsf T$ acts on the left: for $\sigma$ as above we set
    \[\Phi \cdot \sigma \coloneqq (Z(\Phi^{-1}_*(-)),\Phi(\PP)),\]
    where $\Phi_*$ is the automorphism of $K(\Ku(Y))$ induced by $\Phi$.
\end{itemize}

Fix $0<\alpha<\frac 12$. Denote by $\sK$ the $\widetilde{\GL}_2^+(\R)$-orbit of the stability condition $\sigma(\alpha,-\frac 12)$ in $\Stab(\Ku(Y))$. Then we have:

\begin{proposition}[{\cite[Prop 3.6]{PY20}}]
\label{prop_3.6PY}
For all $(\alpha,\beta)\in V$, $\sigma(\alpha,\beta)\in \sK$.
\end{proposition}

Another result of \cite{PY20} is the following:

\begin{proposition}[{\cite[Prop. 5.7]{PY20}}]
\label{lem_sec.5.1PY}
If $Y$ is a Fano threefold of Picard rank 1 and index 2, then there exists $\tilde g \in \widetilde{\GL}_2^+(\R)$ such that
\[ \mathsf R \cdot \sigma(\alpha,-\frac 12) = \sigma(\alpha,-\frac12) \cdot \tilde g. \]
\end{proposition}

For $\sigma\in \sK$ and $\kappa\in N(\Ku(Y))$, we write $M_\sigma(\kappa)$ the moduli space of $\sigma$-stable objects of class $\kappa$ in $\Ku(Y)$. As an immediate consequence of Prop. \ref{lem_sec.5.1PY} we have:
\begin{corollary}
\label{cor_iso_of_moduli}
For all $n\in \Z$, there is an isomorphism
\begin{equation*}
M_\sigma(\kappa)\simeq M_\sigma (\mathsf R^n_*\kappa).
\end{equation*}
\end{corollary}

\section{Lines on a Veronese double cone}
\label{sec_Fano_scheme_of_lines}

\subsection{Veronese double cones}\label{sec_Veronese_double_cones}
We fix some notation and recall some general results on Veronese double cones, following \cite{Iskovskih1977} and \cite{HK15}.
Let $Y$ be a hypersurface cut out by a sextic equation in the weighted projective space $\mathbb P\coloneqq \pr{}(1,1,1,2,3)$. Let $x_0,...,x_4$ be coordinates of $\mathbb P$, where $x_3$ and $x_4$ are those of weight 2,3 respectively. By completing a square, we can write the equation for $Y$ as $x_4^2=f_6(x_0,...,x_3)$ where $f_6$ is a degree 6 polynomial.
The linear series $H\coloneqq  \sO_{\mathbb P}(1)_{|Y}$ has three sections and a unique base point $y_0$ \cite[Prop. 3.1]{Iskovskih1977}, hence it induces a rational map $\phi_H\colon Y \dashrightarrow \mathbb{P}(H^0(\sO_Y(1)))\simeq \pr 2$. On the other hand, $2H\sim -K_Y$ is base point free, and induces a morphism $\phi_{2H}\colon Y \to \mathbb{P}(H^0(\sO_Y(2)))\simeq\pr 6$, whose image $K\simeq \mathbb P(1,1,1,2)$ is the cone over a Veronese surface with vertex $k\coloneqq \phi_{2H}(y_0)$. 

More precisely, for  $V\coloneqq H^0(\sO_Y(1))$ we have 
\[ H^0(\sO_Y(2))=\Sym^2 V \oplus \pair{x_3}, \]
and the map 
\begin{equation}\label{eq_map_i}
    \begin{split}
    i\colon & V\oplus \pair{x_3} \to \Sym^2 V \oplus \pair{x_3}\\
    &\phantom{a}(v,r)\phantom{aa} \mapsto \phantom{aa}(v^2,r)
    \end{split}
\end{equation}
embeds $\mathbb{P}(V\oplus 0)$ as a Veronese surface and identifies the cone over $\mathbb{P}(V\oplus 0)$ and vertex $k=\mathbb{P}(0\oplus \pair{x_3})$ with $K$.

The morphism $\phi_{2H}$ is smooth of degree 2 outside $k$ and the divisor $W\coloneqq \{f_6=0\}\in \abs{\sO_K(3)}$. For this reason, $Y$ is often referred to as to a \textit{Veronese double cone}. We will denote by $\iota$ the involution on $Y$ corresponding to the double cover $\phi_{2H}$.

There is a commutative diagram
\begin{equation}
\label{eq_diagram_rational_maps}
\begin{tikzcd}
	{Y} & {K} \\
	{\pr 2}
	\arrow["{\phi_H}"', from=1-1, to=2-1, dashed]
	\arrow["{\phi_{2H}}", from=1-1, to=1-2]
	\arrow["{\eta}", from=1-2, to=2-1, dashed]
\end{tikzcd}
\end{equation}
where $\eta$ is the projection from $k$. Consider the blowup $\sigma_K\colon \widetilde{K}\to K$ of the vertex $k$ with exceptional divisor $E$. Then, the blow-up $\widetilde{Y} = Y\times_K \widetilde{K}$ resolves the indeterminacy of diagram \eqref{eq_diagram_rational_maps}:
\begin{equation*}
\begin{tikzcd}
	\widetilde{Y} & \widetilde{K} & \\
	{Y} & {K} & {\pr 2}
	\ar[swap, "{\widetilde{\phi}}"', from=1-1, to=1-2]
	\arrow[swap, "{\phi_{2H}}", from=2-1, to=2-2]
	\arrow[swap, "{\sigma_Y}", from=1-1, to=2-1]
	\arrow[swap, "{\sigma_K}", from=1-2, to=2-2]
	\arrow["{\widetilde{\eta}}", from=1-2, to=2-3]
	\arrow[swap, "{\eta}", from=2-2, to=2-3, dashed]
\end{tikzcd}
\end{equation*}
where $\widetilde{\phi}\colon \widetilde{Y}\to \widetilde{K}$ is a degree 2 cover ramified over the divisor $E\cup \sigma_K^{-1}(W)$.

The map $\eta$ restricted to $W$ is a 3-to-1 cover of $\pr 2$, and it ramifies at a curve $C_0$. Throughout this section, we assume that $Y$ is smooth and that $C_0$ is irreducible and general in moduli (this is the generality assumption used in \cite{Tihomirov1982}, whose results we will use).

\subsection{Stable sheaves of class \texorpdfstring{$\kappa_1$}{k1} on \texorpdfstring{$Y$}{Y}}

Let $M_G(v)$ denote the moduli space of stable sheaves of class $v$ on $Y$. Objects in $M_G(v)$ are related to subschemes of $Y$ with Hilbert polynomial $t+1$, we start by studying those.  

\begin{definition}
A \textit{line} in $Y$ is a smooth subscheme of pure dimension 1 with Hilbert polynomial $t+1$.
\end{definition}

In particular, for every line $L$ we have $H.L=1$. We say that the degree of a curve $C\subset Y$ is the integer $H.C$: thus, lines are rational curves of degree 1 in $Y$. 

A similar definition holds for lines and conics in $K$: let $j \colon K \to \pr 6$ the embedding induced by the map $i$ of Eq. \eqref{eq_map_i}. We use the notation $K^\circ \coloneqq K\setminus\{k\}$. 

\begin{definition}[{\cite[Def. 3.1]{HK15}}]
A curve $C$ in $K$ is a \textit{line} (resp. a \textit{conic}) if the closure of its image $j(C\cap K^\circ)$ is a line (resp. a conic) in $\mathbb{P}^6$.
\end{definition}

Lines and conics in $K$ are described in \cite[Sec. 3]{HK15}. Lines in $K$ are the closure of fibers of the projection $\eta\colon K^\circ \to \pr 2$, conics of $K$ are smooth (in which case they do not contain the vertex $k$), or the union of two lines (possibly doubled). For the rest of the section, we use the shorthand $\phi\coloneqq \phi_{2H}$.

\begin{lemma}\label{lem_curves_degree_1}
Let $C$ be a degree 1 curve in $Y$. Then the image $c\coloneqq \phi(C)$ is a conic in $K$, which intersects $W$ in three (possibly coinciding) points with multiplicity 2. There are two possibilities:
\begin{itemize}
    \item[--] $p_a(C)=0$: $c$ is a smooth conic in $K$. In this case, $C$ is a line, and $\phi^{-1}(c)=C\cup C'$ where $C'$ is also a line.
    \item[--] $p_a(C)=1$: $c$ is a doubled line. Then  $C$ is a smooth curve of genus 1, or a singular rational curve.
\end{itemize}
\end{lemma}

\begin{proof}
Since $\phi$ is induced by the linear series $\abs{2H}$, $c=\phi(C)$ must be a conic on $K$. Note first of all that if $c$ is reducible then so is $C$, but this is impossible since $C.H=1$. Hence, $c$ is either smooth or a doubled line. In either case, $c$ cannot be contained in $W$: otherwise $\phi_{|C}$ is an isomorphism since $\phi$ branches over $W$, but this contradicts the assumptions on degree. 

If $c$ is smooth and it intersects $W$ with odd multiplicity at a point, then $\phi^{-1}(c)$ must be irreducible of degree $>1$. This is not the case as $C\subseteq \phi^{-1}(c)$. 
So $c$ is tritangent to $W$, and $\phi^{-1}(c)=C\cup C'$ is the union of two lines. 

If $c=2l$ is a doubled line with $l$ a line in the ruling of $K$, then the restriction of $\phi\colon\phi^{-1}(l)\to l$ is a covering map branched over the four points $(l\cap W) \cup k$. Since $k\notin W$, $\phi^{-1}(l)$ must be irreducible. If the points in $(l\cap W)$ are all distinct, then $C=\phi^{-1}(l)$ is a smooth elliptic curve. If two points of $l \cap W$ coincide, then $C$ has a double point. If all three coincide, $C$ has a cusp.
\end{proof}

We can now classify Gieseker-semistable sheaves of class $\kappa_1$:

\begin{proposition}\label{prop_classification_subschemes}
Semistable sheaves of class $\kappa_1$ on $Y$ are exactly ideal sheaves of subschemes $Z$ with Hilbert polynomial $\chi(\sO_Z(t))=t+1$. There are three possibilities for $Z$:
\begin{enumerate}[label=({\roman*})]
    \item \label{itm:Z_line} $Z$ is a line in $Y$;
    \item \label{itm:Z_embedded_pt} $Z$ is a non-reduced scheme supported on a curve of degree 1 and genus 1 with an embedded point;
    \item \label{itm:Z=CuP} $Z$ is the union of a curve of degree 1 and genus 1 and a point which does not belong to the curve.
\end{enumerate}
\end{proposition}

\begin{proof}
Ideal sheaves are torsion free of rank 1, and therefore stable. So, it suffices to show that a Gieseker-semistable sheaf $E$ of class $\kappa_1$ is an ideal sheaf. This is a standard argument: since $Y$ is smooth, $E \to E^{\vee\vee}$ is injective and $E^\vee$ is reflexive, so that $E^\vee\simeq \sO_Y(-D)$ for some divisor $D$. Therefore $E\otimes \sO_Y(-D)$ is the ideal sheaf of a subscheme supported in codimension 2. Then, $E\simeq I_Z \otimes \sO_Y(D)$, and since $[E]=\kappa_1$ we must have $D=0$ and $\chi(\sO_Z(t))=t+1$ (the Hilbert polynomial is that of $\sO_L$ for $L$ a smooth rational curve in $Y$). 

The three possibilities for $Z$ follow from the fact that $H.Z_{\mathrm{red}}=1$ is the degree of the Hilbert polynomial, so $Z_{\mathrm{red}}$ contains one of the curves described in Lemma \ref{lem_curves_degree_1}. Then, the only possible cases are those listed, note moreover that all three can occur \cite{Tihomirov1982}. 
\end{proof}

We will refer to the three possibilities listed in Proposition \ref{prop_classification_subschemes} as to subschemes of type \ref{itm:Z_line}, \ref{itm:Z_embedded_pt}, and \ref{itm:Z=CuP}. Observe moreover that Proposition \ref{prop_classification_subschemes} implies the following: 

\begin{proposition}
\label{cor_Hilb=M_G}
The moduli space $M_G(\kappa_1)$ is isomorphic to the Hilbert scheme of lines $\Hilb(Y,t+1)$. 
\end{proposition}

\begin{proof}
We argue as in the proof of \cite[Theorem 2.7]{PT09_CurveCounting}. Let $\sI$ be a flat family of semistable sheaves of class $\kappa_1$ over a base $B$, normalized so that it has trivial determinant along $B$. The sheaf $\sI$ has rank 1, and it is pure so it injects into its double dual 
\[ 0\to \sI \to \sI^{\vee\vee}. \]
Flatness of $\sI$ implies that $\sI^{\vee\vee}$ is locally free, and $\sI^{\vee\vee}$ has trivial determinant since $\sI$ does. Therefore, $\sI^{\vee\vee}\simeq \sO_{Y\times B}$, and there is a short exact sequence 
\[0\to \sI \to \sO_{Y\times B} \to \sQ \to 0, \]
where $\sQ$ is a flat family of quotients of $\sO_{Y\times B}$. Conversely, any such family of quotients gives rise to a family of ideal sheaves as those listed in Prop. \ref{prop_classification_subschemes}. 

This identifies the functors represented by $M_G(\kappa_1)$ and $\Hilb(Y,t+1)$.
\end{proof}

\begin{remark}
As mentioned in the introduction, $3H$ is the smallest very ample multiple of $H$. The embedding $Y \to \mathbb{P}(H^0(\sO_Y(3H)))$ maps the Hilbert scheme $\Hilb(Y,t+1)$ to that of twisted cubics, which has two irreducible components whose intersection parametrizes non-reduced subschemes \cite[Sec. 3]{CK11}.
\end{remark}

We describe the Gieseker moduli space $M_G(\kappa_1)$. We prove the theorem here, even if in the proof we apply Proposition \ref{prop_ext1=5_non_reduced_smooth}, which is postponed to after some more technical computations:

\begin{thm}\label{thm_Gieseker}
The moduli space $M_G(\kappa_1)$ has two irreducible components $M_1$ and $M_2$. 

$M_1$ is a smooth surface compactifying the locus of ideals of smooth lines of $Y$. 
$M_2$ has dimension 5, and its general object is a subscheme of type \ref{itm:Z=CuP}. It is smooth outside the intersection with $M_1$. 

Points in $M_1\cap M_2\simeq C_0$ parametrize singular rational curves with an embedded point at the singularity. 
\end{thm}

\begin{proof}
The component $M_1$ parametrizing ideal sheaves of lines is described in \cite[Theorem 4]{Tihomirov1982}: $M_1$ is a smooth surface intersecting the rest of $M_G(v)$ on the locus parametrizing singular curves with a nilpotent embedded at the singularity. This locus is isomorphic to the curve $C_0$.  

There is a 5 dimensional family of schemes of type \ref{itm:Z=CuP} (two parameters determine the one dimensional component, and three determine the point). Denote by $M_2$ the component of $M_G(v)$ containing this family. By Prop. \ref{cor_Hilb=M_G}, the tangent space at $Z=C\cup p$ of type \ref{itm:Z=CuP} is 
\[ T_Z M_2\simeq \Hom(I_{Z},\sO_{Z})\simeq \Hom(I_C,\sO_C)\oplus \Hom(I_p,\sO_p) \]
The spaces in the right hand side parametrize deformations of $C$ and $p$ respectively, so $\dim T_Z M_2=5$. This shows that $\dim M_2=5$ and that $M_2$ is smooth at type \ref{itm:Z=CuP} points. Moreover, Proposition \ref{prop_ext1=5_non_reduced_smooth} shows that $M_2$ is smooth at points of type \ref{itm:Z_embedded_pt} for which the nilpotent is supported on smooth points. 

Finally, there are no other components in $M_G(v)$, because we exhausted the possibilities in Proposition \ref{prop_classification_subschemes}.
\end{proof}

\begin{remark}\label{rmk_def_of_F(Y)}
The component $M_1$ is sometimes denoted $F(Y)$ and called the \textit{Fano surface of lines} of $Y$ (e.g. in \cite{Tihomirov1982}).
\end{remark}

\begin{lemma}
\label{lem_exti(IC,OC)}
Let $C$ be a curve in $Y$ of degree 1 and arithmetic genus $p_a=1$. Then 
$$\begin{cases}\Ext^0(I_C,I_C)=\C   \\
  \Ext^1(I_C,I_C)=\C^2 \\ 
  \Ext^i(I_C,I_C)=0 \mbox{ otherwise.}
\end{cases}$$

\end{lemma}

\begin{proof}
The curve $C$ is cut out by the pull-back of two linear forms from $\pr 2$ via $\eta\colon K^\circ \to \pr 2$, denote them $l,m$. In fact, the Koszul complex in $l$ and $m$ is exact on $Y$:
\begin{equation}
    \label{eq_Koszul_resolution_Zred}
    0\to \sO_Y(-2) \xrightarrow{\begin{pmatrix}m\\-l\end{pmatrix}} \sO_Y(-1)^{\oplus 2} \to I_C \to 0.
\end{equation}
Applying the functor $\Hom(-,I_C)$ gives the map:
\begin{equation}
    \label{eq_Koszul_Hom}
\end{equation}
$$H^0(I_C(1))^{\oplus 2}\simeq \Hom(\sO_Y(-1),I_C)^{\oplus 2} \xrightarrow{\cdot \begin{pmatrix} m & -l \end{pmatrix}} \Hom(\sO_Y(-2),I_C)\simeq H^0(I_C(2))$$. 

It is straightforward to check that the map \eqref{eq_Koszul_Hom} has rank 3, and the conclusion follows. 
\end{proof}

\begin{lemma}
\label{lem_exti(IZ,Op)}
Let $Z$ be a subscheme of type \ref{itm:Z_embedded_pt} with the embedded point $p$ in the smooth locus of $Z_{\mathrm{red}}$. Then
$$\Ext^i(I_Z,\sO_p)= \begin{cases} \C^3 & \mbox{ if }i=0,1 \\
\C & \mbox{ if }i=2\\
0 & \mbox{ otherwise.}
\end{cases}$$
Moreover, applying $\Hom(-,\sO_p)$ to the sequence
\begin{equation}
\label{eq_ses_IZ_IZred_Op}
   I_Z\to I_{Z_{\mathrm{red}}} \to \sO_p, 
\end{equation}
we get a non-zero homomorphism $\alpha\colon \Ext^1(\sO_p,\sO_p) \to \Ext^1(I_{Z_{\mathrm{red}}},\sO_p)$.
\end{lemma}

\begin{proof}
The groups $\Hom^\ast(\sO_p,\sO_p)$ are the exterior algebra on the tangent space at $p$, so they have dimensions 1,3,3,1 for $\ast=0,1,2,3$. 
Applying the functor $\Hom(-,\sO_p)$ to the resolution \eqref{eq_Koszul_resolution_Zred} as in Lemma \ref{lem_exti(IC,OC)}, we see that $\mathrm{hom}^*(I_{Z_{\mathrm{red}}},\sO_p)=2,1,0,0$ for $*=,0,1,2,3$.

Apply $\Hom(-,\sO_p)$ to the sequence \eqref{eq_ses_IZ_IZred_Op} and consider the corresponding long exact sequence: this shows immediately that 
\[ \mathrm{ext}^2(I_Z,\sO_p)=1  \qquad  \mathrm{ext}^3(I_Z,\sO_p)=0. \]

On the other hand, we may consider a set of local coordinates around $p$ given as $\{l,m,s\}$, where $l,m$ define $Z_{\mathrm{red}}$. Then, $l,m^2$, and $m s$ generate $I_Z$ locally around $p$. Resolving $I_Z$ using these generators we see that $\mathrm{hom}(I_Z,\sO_p)=3$, arguing as above.

Finally, observe that $\chi(I_Z,\sO_p)=\chi(I_Z,\sO_q)=\chi(\sO_Y,\sO_q)=1$ where $q\in Y\setminus Z_{\mathrm{red}}$ (since this quantity only depends on the numerical class of $\sO_p$), which implies that $\mathrm{ext}^1(I_Z,\sO_p)=3$.

The map $\alpha$ appears in the long exact sequence, and a simple dimension count shows that it does not vanish.
\end{proof}

\begin{proposition}\label{prop_ext1=5_non_reduced_smooth}
If $Z$ is a subscheme of type \ref{itm:Z_embedded_pt} with the embedded point in the smooth locus of $Z_{\mathrm{red}}$, then
\[ \mathrm{ext}^1(I_Z,I_Z)=5. \]
\end{proposition}

\begin{proof}
We may write $I_Z\simeq [I_{Z_{\mathrm{red}}} \to \sO_p]$ where $p$ is the embedded point. Then, $R\Hom( I_Z, I_Z)$ may be computed with the spectral sequence 
\begin{equation}\label{eq_spec_seq}
E_1^{p,q}=H^q(K^{\bullet,p})\Rightarrow H^{p+q}(K^\bullet).
\end{equation}
The first page is

\begin{center}
\begin{tabular}{c|cc}
$\vdots$ & $\vdots$ & $\vdots$\\
 $\Ext^1(\sO_p,I_{Z_{\mathrm{red}}})$ & $\Ext^1(I_{Z_{\mathrm{red}}},I_{Z_{\mathrm{red}}})\oplus \Ext^1(\sO_p,\sO_p)$ &$\Ext^1(I_{Z_{\mathrm{red}}},\sO_p)$\\
$\Hom(\sO_p,I_{Z_{\mathrm{red}}})$ & $\Hom(I_{Z_{\mathrm{red}}},I_{Z_{\mathrm{red}}})\oplus \Hom(\sO_p,\sO_p)$ &$\Hom(I_{Z_{\mathrm{red}}},\sO_p)$\\
\hline
($p=-1$)&($p=0$)&($p=1$)
\end{tabular} 
\end{center}
with arrows pointing to the right and zeros in all other columns. We claim that the dimensions of the vector spaces above are given by
\begin{center}
\begin{tabular}{c|cc}
2 & 1 &0\\
1 & 3 &0\\
0 & $2+3$ &1\\
0 & $1+1$ &2\\
\hline
\end{tabular}
\end{center}
Indeed, the third column (and hence, by Serre duality, the first one) is computed in the proof of Lemma \ref{lem_exti(IZ,Op)}. 

The contributions from $\Hom^\bullet(I_{Z_{\mathrm{red}}},I_{Z_{\mathrm{red}}})$ in the central column follow from Lemma \ref{lem_exti(IC,OC)}, while the dimensions of $\Hom^\bullet(\sO_p,\sO_p)$ follow because $p$ is a smooth point of $Y$, as in the proof of Lemma \ref{lem_exti(IZ,Op)}.

Our next claim is that the maps in the middle rows are non-zero, and that the map in the bottom row has one-dimensional image. Granting the claim, the second page of the spectral sequence reads
\begin{center}
\begin{tabular}{c|cc}
$\ast$ & $\ast$ &0\\
0 & 2 &0\\
0 & 4 & 0\\
0 & 1 & 1\\
\hline
\end{tabular}
\end{center}
and hence $\mathrm{ext}^1(I_Z,I_Z)=5$.

The map on the second row from the top is $\Ext^2(\sO_p,I_{Z_{\mathrm{red}}}) \to \Ext^2(\sO_p,\sO_p)$. It is Serre dual to the homomorphism $\alpha$ (see Lemma \ref{lem_exti(IZ,Op)}), which is also the restriction to the second summand of the map on the third row:
\[\Ext^1(I_{Z_{\mathrm{red}}},I_{Z_{\mathrm{red}}})\oplus \Ext^1(\sO_p,\sO_p) \to \Ext^1(I_{Z_{\mathrm{red}}},\sO_p). \]
It follows from Lemma \ref{lem_exti(IZ,Op)} that these two maps do not vanish. 
Finally, observe that the map 
\[ \Hom(I_{Z_{\mathrm{red}}},I_{Z_{\mathrm{red}}})\oplus \Hom(\sO_p,\sO_p) \to \Hom(I_{Z_{\mathrm{red}}},\sO_p) \]
has one-dimensional image (the span of the natural map $I_{Z_{\mathrm{red}}}\to \sO_p$ of \eqref{eq_ses_IZ_IZred_Op}).
\end{proof}

\section{Moduli spaces of objects of \texorpdfstring{$\Ku(Y)$}{Ku(Y)}}\label{sec_obj_and_mod_in_ku}

For the rest of this note, $Y$ will denote a general Veronese double cone (we will follow the notation of Section \ref{sec_Veronese_double_cones}). When a result holds for all Fano threefolds of Picard rank 1 and index 2, we will make it explicit.
In this section, we construct three families of objects of $\Ku(Y)$ and show that they are related by a rotation. More precisely, we show that the set $\{ \pm\kappa_1,\pm\kappa_2,\pm(\kappa_1-\kappa_2) \}$ is an orbit of the action of $\mathsf R_*$ on $N(\Ku(Y))$. 

As a result, Corollary \ref{cor_iso_of_moduli} yields an isomorphism of the corresponding moduli spaces. 

We start by defining the three families of objects:

\begin{enumerate}[label=\textbf{(\Alph*)}]
    \item \label{item_Mp}  For any Fano threefold $Y$ of Picard rank one, index 2, and degree $d$, we can consider projections of skyscraper sheaves to $\Ku(Y)$: for $p\in Y\setminus \{y_0\}$, the projection $\pi(\C_p)$ of $\C_p$ is the complex $M_p[1]$, defined as the cone
\begin{equation}\label{eq_def_of_Mp}
     \sO_Y^{d+1} \to I_p(1)\to M_p.
\end{equation}
The projection of $y_0$ on the other hand, is defined by
    \begin{equation*}\label{eq:def_of_My0}
        \sO_Y^3\oplus \sO_Y[-1]\to I_{y_0}(1)\to M_{y_0}.
    \end{equation*}
We have $[M_p]=\kappa_2-d\kappa_1$.
\item \label{item_Ep} A second family of objects are the complexes $E_p$ studied in \cite{APR19}. They have class $\kappa_2$, and are defined by the distinguished triangle 
\[ \sO_Y(-1)[1] \to E_p \to I_p \]
for any point $p\in Y$.
\item \label{item_FpGx} Assume now that $Y$ has degree 1. Then, we can construct another class of objects as follows. For a point $p\in Y\setminus \{y_0\}$, let $x\coloneqq \phi_H(p)\in \pr 2$ and let $C\coloneqq C_x$ be the corresponding genus 1 curve (notation as in Sec. \ref{sec_Fano_scheme_of_lines}). Then, $H^0(\sO_C(p))=\C$, and we consider the cone of the triangle
\begin{equation}
    \label{eq_def_Fp}
    \sO_Y \to \sO_C(p) \to F_p.
\end{equation}
Similarly, define complexes associated with $y_0$: for all $x\in \pr 2$, $y_0\in C_x$ and $H^0(\sO_{C_x}(y_0))=\C$ as above, so we write 
\begin{equation}
    \label{eq_def_Gx}
    \sO_Y \to \sO_{C_x}(y_0) \to G_{x}
\end{equation}
for the corresponding cones. 
\end{enumerate}

\begin{remark}
\label{rmk_Gx_notin_Ku}
\begin{itemize}
    \item[--] The numerical class of $F_p$ and $G_x$ is $-\kappa_1$. In fact, $\sO_C(p)$ (and $\sO_{C_x}(y_0)$) has the same Hilbert polynomial as $\sO_\ell$ for any line $\ell\subset Y$, so $[F_p]=[G_x]=-[I_\ell]=-\kappa_1$;
    \item[--] The objects $F_p$ belong to $\Ku(Y)$: the vanishing $\Hom(\sO_Y(1),F_p)=0$ follows from \eqref{eq_def_Fp}  and the observation that the sheaves $\sO_Y(-1)$ and $\sO_C(p-y_0)$ have no cohomologies. Similarly, the isomorphism $\mathbf R\Hom(\sO_Y,\sO_Y)\simeq \mathbf R\Hom(\sO_Y,\sO_C(p))$ implies the vanishing of $\Hom(\sO_Y,F_p)$.
    \item[--] On the other hand, the objects $G_x\notin \Ku(Y)$. 
    Note, in fact, that for any curve $C_x$ we have  
\[ \sO_{C_x}\otimes \sO_Y(1)\simeq \sO_{C_x}(y_0), \]
since $C_x$ is defined by two linear forms, and a third one will intersect $C_x$ precisely at the base locus of $\abs{\sO_Y(1)}$, which is $y_0$.
Then, by \eqref{eq_def_Gx} we have
\[\Hom(\sO_Y(1),G_x)\simeq \Hom(\sO_Y(1),\sO_{C_x}(y_0)) \simeq \Hom (\sO_Y,\sO_{C_x}(y_0-y_0))=\C.\]
\end{itemize}
\end{remark}

 The three classes of objects \ref{item_Mp}, \ref{item_Ep}, and \ref{item_FpGx} are related by rotations:

\begin{lemma}\label{lem_R(Ep)=Mp}
We have $\mathsf R(E_p)=M_p$ for every $p\in Y$. This holds for $Y$ of any degree. 
\end{lemma}

\begin{proof}
Twist the defining sequence of $E_p$:
\[ \sO_Y[1] \to E_p(1) \to I_p(1) \]
and mutating across $\sO_Y$ shows $\mathsf R(E_p)\simeq \mathbb L_{\sO_Y}(I_p(1))$. Then, observe that \eqref{eq_def_of_Mp} computes $\mathbb L_{\sO_Y}(I_p(1))$.
\end{proof}

Recall that $\iota\colon Y\to Y$ is the involution corresponding to the double cover $\phi_{2H}\colon Y\to K$. Then we have: 

\begin{lemma}
\label{lem_R(Mp)=Fip}
For $p\neq y_0$, we have $\mathsf R(M_p)=F_{\iota(p)}$.
\end{lemma}

\begin{proof}
By its definition, the cohomologies of $M_p(1)$ are those of the complex $[\sO_Y^2(1)\xrightarrow{ev} I_p(2)]$. The kernel of the evaluation map is $\sO_Y$, and the cokernel is the cokernel of the inclusion $I_C(2)\to I_p(2)$, which is $\sO_C(2y_0-p)$,
where $C\coloneqq C_{\phi_H(p)}$. 
This shows that $\mathsf R(M_p)=\mathbb L_{\sO_Y}(\sO_C(2y_0-p))$. One then checks that the divisor $2y_0-p$ on $C$ is linearly equivalent to $\iota(p)$, by considering the Weierstrass equation for $C$ in $\mathbb P(1,1,1,2,3)$ and observing that taking inverses coincides with applying $\iota_{|C}$. Therefore, $\mathsf R(M_p)=F_{\iota(p)}$.
\end{proof}

From Lemmas \ref{lem_R(Ep)=Mp} and \ref{lem_R(Mp)=Fip} we get:

\begin{corollary}\label{cor_matrix_of_R*}
The matrix associated to $\mathsf R_*$ in the basis $\kappa_1,\kappa_2$ is  $ \begin{pmatrix}
   0& -1\\1&1
\end{pmatrix}. $
In particular, $\mathsf R_*$ acts transitively on the set $\{\pm\kappa_1,\pm\kappa_2,\pm(\kappa_1-\kappa_2)\}$ of classes in $N(\Ku(Y))$ with square $-1$.
\end{corollary}

\begin{proof}
By Lemma \ref{lem_R(Ep)=Mp}, we have $\mathsf R_*(\kappa_2)=\kappa_2-\kappa_1$, and by Lemma \ref{lem_R(Mp)=Fip} we have $\mathsf R_*(\kappa_2-\kappa_1)=-\kappa_1$, the rest is straightforward.
\end{proof}

\begin{remark}[Homological dimension]\label{rmk_hom_dim}
The heart $\sA(\alpha,\beta)$ has homological dimension 2 if $d=2,3$ \cite{PY20}. This is false in the case $d=1$. In fact, by Lemmas \ref{lem_R(Ep)=Mp} and \ref{lem_R(Mp)=Fip} above we have $E_{\iota(p)}\simeq \mathsf R^{-2}(F_p)$ for $p\neq y_0$ in $Y$. Then, by Serre duality and Lemma \ref{lem_serre_dual},
\[ \Ext^3(F_p,E_{\iota(p)})\simeq \Hom(F_{p},\mathsf R^{-2}(F_{p})[3]) \simeq \Hom(F_{p},F_{p})^*\neq 0. \]
\end{remark}

We now recollect the results of this section in the following theorem (we use the same notation $M_1$ for the copy of $F(Y)$ embedded as an irreducible component in $M_G(\kappa_1)$ (Theorem \ref{thm_Gieseker}) and in $M_\sigma(-\kappa_1)$).

\begin{thm}\label{thm_three_moduli}
Let $Y$ be a general smooth Veronese double cone, and let $\sigma$ be a stability condition in $\mathcal K$ (see Section \ref{sec_KuzComp}). The moduli spaces $M_\sigma(-\kappa_1), M_\sigma(-\kappa_2)$ and $M_\sigma(\kappa_1-\kappa_2)$ are isomorphic. They have two irreducible components $M_1$ and $M_3$ isomorphic respectively to the Fano surface of lines $F(Y)$ and to $Y$ itself, intersecting along $C_0$. The generic point of the component $Y$ parameterizes, respectively, objects of form $F_p$, $E_p$, and $M_p$.
\end{thm}

\begin{proof}
Corollaries \ref{cor_iso_of_moduli} and \ref{cor_matrix_of_R*} yield the isomorphism of moduli spaces. The description of the irreducible components is \cite[Theor. 1.2]{APR19}. The statement on the general objects follows again from Lemmas \ref{lem_R(Ep)=Mp} and \ref{lem_R(Mp)=Fip}. 
\end{proof}

We conclude the section describing the objects $E_{y_0}, M_{y_0}$, and $F_{y_0}\coloneqq \mathsf R^2(E_{y_0})$: these correspond to the point $y_0$ in the component $Y$ of the three moduli spaces of Theorem \ref{thm_three_moduli}, and they are of a different nature from the others.

\begin{proposition}[Rotations at $y_0$]\label{lem_rot_at_k}
We have $\mathsf R(E_{y_0})=M_{y_0}$, a complex with cohomologies 
\begin{equation*}
    \begin{split}
         \sH^{-1}(M_{y_0})&\simeq \coker(\sO_Y(-2)\to \sO_Y(-1)^{\oplus 3})\\
         \sH^0(M_{y_0})&=\sO_Y.
    \end{split}
\end{equation*}
The complex $F_{y_0}$ has three cohomologies, and it fits in a triangle
\begin{equation}\label{eq_destab_tr_Fk}
     \sO_Y(-1)[2]\to F_{y_0}\to [\sO_Y^{\oplus 3} \to \sO_Y(1)].
\end{equation}
\end{proposition}

\begin{proof}
One shows as above that $\mathsf R(E_{y_0})=M_{y_0}$, the (shift of the) projection of the skyscraper sheaf $\C_{y_0}$ to $\Ku(Y)$.
All hyperplane section of $Y$ pass through $y_0$. In other words, $M_{y_0}$ is defined by an exact triangle
\[ \sO_Y^{\oplus 3}\oplus \sO_Y[-1] \to I_{y_0}(1) \to M_{y_0} \]
whose cohomology sequence is
\begin{equation}
    \label{eq_cohomology_of_M_k} 0\to \sH^{-1}(M_{y_0}) \to \sO_Y^{\oplus 3} \xrightarrow{ev} I_{y_0}(1) \to \sH^0(M_{y_0}) \to \sO_Y \to 0, 
\end{equation}
where the evaluation map $ev$ is surjective, and coincides with the last map of a Koszul complex on three linear forms. Therefore,  $\sH^{-1}(M_{y_0})\simeq \coker(\sO_Y(-2)\to \sO_Y(-1)^{\oplus 3})$ and $\sH^0(M_{y_0})=\sO_Y$.

To compute $F_{y_0} = \mathsf R(M_{y_0})$, compute the cohomology sheaves of $M_{y_0}(1)$ by twisting \eqref{eq_cohomology_of_M_k}, and write the cohomology sequence of the triangle
\[ \sO_Y^{\oplus 3} \oplus \sO_Y^{\oplus 3}[1] \to M_{y_0}(1) \to F_{y_0}.  \]
It reads
\[ 0\to \sO_Y(-1) \to \sO_Y^3 \to \coker (\sO_Y(-1)\to \sO_Y^{\oplus 3}) \xrightarrow{0}  \sH^{-1}(F_{y_0}) \to \sO_Y^3 \to \sO_Y(1) \to \C_{y_0} \to 0,\]
whence the claim.
\end{proof}

\section{Set-theoretic considerations}\label{sec_set_theoretic}

\subsection{Stable complexes of class \texorpdfstring{$\kappa_1$}{k1}}
\label{sec_tilt_stability}

In this section, we classify objects of class $\kappa_1$ that are semistable with respect to $\sigma^0_{\alpha,\beta}$ and $\sigma_{\alpha,\beta}$. Here, $\sigma$ denotes one of the stability conditions of Theorem \ref{thm_stab_cond_on_Ku}. 

Our classification shows that following the strategy of \cite{APR19} and \cite{PY20} to describe $M_\sigma(\kappa_1)$ is more difficult in this setting. In those works, moduli spaces of $\sigma$-stable objects are related via wall-crossing to moduli spaces of complexes which are stable with respect to $\sigma_{\alpha,\beta}$ and $\sigma^0_{\alpha,\beta}$. More precisely, for $v=\kappa_2$, or $d>1$ and $v=\kappa_1$, the three notions of stability coincide, and we have
\[ M_\sigma(v) \simeq M_{\sigma_{\alpha,\beta}^0}(v) \simeq M_{\sigma_{\alpha,\beta}}(v) \]
(this is also the case for cubic fourfolds, \cite{BLMS17}). If $d=1$ and $v=\kappa_1$, there are objects in $D^b(Y)$ that are $\sigma$-semistable but not $\sigma_{\alpha,\beta}$-semistable, and conversely. We will show:

\begin{thm}
\label{thm_classification}
Let $E$ be a complex in $D^b(Y)$ of class $-\kappa_1$, fix $\beta=-\frac 12$ and $\alpha\ll 1$. Then:
\begin{enumerate}[label=(\arabic*)]
    \item \label{itm_Sab} $E$ is $\sigma_{\alpha,\beta}$-semistable if and only if it is a Gieseker stable sheaf in $M_G(\kappa_1)$ (classified in Prop. \ref{prop_classification_subschemes});
    \item \label{itm_Sab0} $E$ is $\sigma^0_{\alpha,\beta}$-semistable if and only if $E$ is isomorphic to:
     \begin{enumerate}[label=(\roman*)]
         \item $F_p$, for $p\neq y_0$,  \label{itm_Fp_Sab0} 
         \item  $G_x$, for $x\in \pr 2$, or \label{itm_Gx_Sab0} 
         \item $I_\ell[1]$, where $\ell\subset Y$ is a line. \label{itm_IZ_Sab0}
     \end{enumerate}
\end{enumerate}
\end{thm}

 We start the proof with some lemmas computing $\sigma_{\alpha,\beta}$-walls in the $(\alpha,\beta)$-plane for $-\kappa_1$. Observe that, by definition of $Z_{\alpha,\beta}^0$ (see Eq. \eqref{eq_def_of_Z^0}), the same equations define numerical walls for both weak stability conditions $\sigma_{\alpha,\beta}$ and $\sigma_{\alpha,\beta}^0$.
 
\begin{lemma}\label{lem_F_p_mu_b-stable}
For $\beta=0$, objects $F_p$ and $G_x$ are strictly semistable of infinite slope in $\Coh^\beta(Y)$. In other words, the half-line $\beta=0$ is a vertical wall for $-\kappa_1$ in the $(\alpha,\beta)$-plane.
\end{lemma}
\begin{proof}
The complex $F_p$ fits into the exact triangle
$$ \sO_C(p)\to F_p\to \sO_Y[1].$$
Both $\sO_C(p)$ and $\sO_Y[1]$ are semistable of infinite slope in $\Coh^\beta(Y)$: it is straightforward to compute that $\Im Z_{\alpha,\beta}(-)$ vanishes on both $\sO_Y$ and $\sO_C(p)$ since 
\[ \Ch_{\leq 2}^{\beta=0}(\sO_Y) = (1,0,0) \qquad \mbox{ and } \qquad \Ch_{\leq 2}^{\beta=0}(\sO_C(p)) = (0,0,H^2). \qedhere \]
\end{proof}

\begin{lemma}\label{lem_largest_Wall}
There are no actual walls for $-\kappa_1$ in the strip $-1<\beta<0$. 
\end{lemma}

\begin{proof}
By Lemma \ref{lem_F_p_mu_b-stable}, the line $\beta=0$ is a vertical wall. 

Next, we show that no actual walls intersect the line $\beta=-1$. 
Suppose otherwise that for some $\alpha>0$ there is an actual wall, realized by a sequence of $\sigma_{\alpha,-1}$-semistable complexes 
\begin{equation}\label{eq_Wall}
    0\to E\to F\to G\to 0
\end{equation} in $\Coh^{-1}(Y)$. 
Observe that for any $\alpha>0$ and any $F$ of class $-\kappa_1$,  $\Im Z_{\alpha,-1}(F)=1$ is the smallest positive value of
$$\Im Z_{\alpha,-1}(-)=H^2\Ch^{-1}_1(-).$$
Then, either $\Im Z_{\alpha,-1}(E)=1$, and therefore $Z_{\alpha,{-1}}(G)=0$, or $\Im Z_{\alpha,-1}(G)=1$ and $Z_{\alpha,{-1}}(E)=0$. Assume the former: since $G$ is $\sigma_{\alpha,{-1}}$-semistable, the support property implies 
$\Ch_{\leq 2} G=0$, which means that \eqref{eq_Wall} is not an actual wall. The same argument works in the latter case swapping the roles of $E$, $G$. 

By Theorem \ref{thm_Nested_Wall}, walls are nested semicircles in the $(\alpha,\beta)$-plane. Therefore it suffices to find a semicircular wall outside the strip $-1<\beta<0$. A standard computation (sketched below for the ease of reading) shows that the class $[\sO_Y(-1)]$ defines a numerical wall on the semicircle with radius $\frac{1}{2}$ and center $(0,-\frac{3}{2})$
We have:
\[ \Ch_{\leq 2}^\beta(\sO_Y(-1))= \left(1,-H-\beta H, \frac{H^2}{2}+\beta H^2 + \frac{\beta^2}{2}H^2  \right), \]
\[ Z_{\alpha,\beta}(\sO_Y(-1))=\left( \frac{\alpha^2-\beta^2}{2} -\frac{1}{2}-\beta \right)- i \left(1+\beta \right), \]
\[ Z_{\alpha,\beta}(F)=\left( \frac{\alpha^2-\beta^2}{2} +1 \right)- i\beta. \]

Then, the condition that $\mu_{\alpha,\beta}(\sO_Y(-1)) = \mu_{\alpha,\beta}(F)$ simplifies to 
\[ \alpha^2 + \left( \beta + \frac 32\right)^2 = \frac 14,\]
the desired semicircle.
\end{proof}

\begin{lemma}\label{lem_F_p_stable}
Objects $F_p$ and $G_x$ are $\sigma^0_{\alpha,\beta}$-semistable for $\alpha>0, -1<\beta<0$. 
\end{lemma}
\begin{proof}
Lemma \ref{lem_F_p_mu_b-stable}  implies that $F_p$ is $\sigma_{\alpha,0}^0$-semistable in $\Coh^0_{\alpha,0}(Y)$ of slope $0$, arguing as in the proof of \cite[Lemma 2.16]{BLMS17}\footnote{A similar argument is used to show \cite[Prop. 4.1]{FP21}, which directly applies to this case and implies that $F_p$ is $\sigma_{\alpha,0}^0$-semistable.}.

Suppose for the moment that $-1\ll\beta<0$ and $\alpha>-\beta$. Then we have $\sO_Y[1],\sO_C(p)\in \Coh^0_{\alpha,\beta}(Y)$, and therefore $F_p\in \Coh^0_{\alpha,\beta}(Y)$ (although $F_p\notin \Coh^\beta(Y)$ since $\sO_Y[1]\notin \Coh^\beta(Y)$). 
Since 
$$\mu^0_{\alpha,\beta}(\sO_Y[1])>0=\mu^0_{\alpha,\beta}(\sO_C(p)),$$
$F_p$ is $\sigma^0_{\alpha,\beta}$-semistable. 

Since walls for $\sigma^0_{\alpha,\beta}$-stability coincide with those for tilt-stability, $F_p$ remains semistable in the region left of the vertical wall $\beta=0$ and outside of the largest  semicircular wall of Lemma \ref{lem_largest_Wall}. In particular, $F_p$ is $\sigma^0_{\alpha,\beta}$-semistable for all $-1<\beta<0$ and all $\alpha >0$. 
\end{proof}

\begin{proposition}\label{prop_stable_objects1}
Let $\beta=-\frac 12$ and $\alpha\ll 1$. Then, the objects listed  
in Theorem \ref{thm_classification} \ref{itm_Sab0} are $\sigma^0_{\alpha,\beta}$-semistable.
\end{proposition}
\begin{proof}
The same argument as \cite[Prop. 4.1]{PY20} applies to the $I_\ell[1]$ and implies that they are $\sigma^0_{\alpha,\beta}$-semistable.
The other objects are $\sigma^0_{\alpha,\beta}$-semistable by Lemma \ref{lem_F_p_stable}.
\end{proof}

\begin{remark}\label{rmk_IZ1_unstable}
If $Z$ is one of the subschemes listed in Prop. \ref{prop_classification_subschemes}, but not a line, the argument of \cite[Prop. 4.1]{PY20} still applies. However, it only implies that $I_Z$ is $\sigma_{\alpha,\beta}$-semistable. In fact, in cases \ref{itm:Z_embedded_pt} and \ref{itm:Z=CuP} of Prop. \ref{prop_classification_subschemes}, $I_Z[1]$ fits into an exact triangle
\begin{equation}
\label{eq_wall_crossing_CpIC}
 \C_p \to I_Z[1] \to I_C[1],
\end{equation}
where $C$ is a degree 1 genus 1 curve, and $p$ is a point embedded in (respectively, disjoint from) $C$.
This is a destabilizing sequence in $\sigma^0_{\alpha,\beta}$\footnote{It is noteworthy that this is a class of objects to which \cite[Prop. 4.1]{FP21} does not apply.}. 

The tradeoff  between $\sigma_{\alpha,\beta}$-stable and $\sigma^0_{\alpha,\beta}$-stable objects behaves like a wall-crossing. Indeed, the extension group $\Ext^1(\C_p,I_C[1])$ in the direction opposite to \eqref{eq_wall_crossing_CpIC} vanishes for disconnected schemes (type \ref{itm:Z=CuP} in Prop. \ref{prop_classification_subschemes}) and is one-dimensional for non-reduced schemes (as computed in the proof of Prop. \ref{prop_ext1=5_non_reduced_smooth}). The corresponding non-trivial extensions are precisely the objects $F_p$ (if $p\neq y_0$) and $G_x$ (if $p=y_0$).
There are no evident walls crossed by the rotation from $\sigma_{\alpha,\beta}$ to $\sigma^0_{\alpha,\beta}$, but it should be possible to find a corresponding wall in the stability manifold of $Y$.
\end{remark}

Next, we show that the objects listed in Theorem \ref{thm_classification} are the only $\sigma^0_{\alpha,\beta}$-semistable objects.

\begin{proposition}\label{prop_semistable_objects2}
Let $\beta=-\frac 12$ and $\alpha \ll  1$. Suppose $F$ is $\sigma^0_{\alpha,\beta}$-semistable object of class $-\kappa_1$. Then $F$ is one of the objects \ref{itm_Fp_Sab0}, \ref{itm_Gx_Sab0} in Theorem \ref{thm_classification}\ref{itm_Sab0}, or $F=I_Z[1]$ where $Z$ is a subscheme as in Prop. \ref{prop_classification_subschemes}.
\end{proposition}
\begin{proof}
Follows from lemmas \ref{lem_coh_to_coh0} and \ref{lem_stable_k1} below.
\end{proof}
\begin{lemma}\label{lem_coh_to_coh0}
For $F$ as in Prop. \ref{prop_semistable_objects2},  there is a triangle
$$F'[1]\to F\to T$$
where $F'\in \Coh^\beta(Y)$ is $\sigma_{\alpha,\beta}$-semistable, and $T$ is either $0$ or $\C_p$ for some $p\in Y$.
\end{lemma}
\begin{proof}
Since $F$ is in $\Coh^0_{\alpha,\beta}(Y)$, there is a triangle
$$F'[1]\to F\to T$$
with $F'\in \Coh^\beta(Y)_{\mu_{\alpha,\beta}\leq 0}$, $T\in \Coh^\beta(Y)_{\mu_{\alpha,\beta}> 0}$. Since $F$ is semistable with respect to $\mu^0_{\alpha,\beta}$, $Z_{\alpha,\beta}(T)$ has to be $0$. So, $T$ is supported on points, that is, $T$ has finite length $m$
and hence $Z_{\alpha,\beta}(F)=Z_{\alpha,\beta}(F'[1])$.
On the other hand, we claim that $F'$ must be $\sigma_{\alpha,\beta}$-semistable. Otherwise, suppose it is destabilized by some $S\subset F'$ in $\Coh^\beta(Y)$. Since  $\Coh^\beta(Y)_{\mu_{\alpha,\beta}\leq 0}$ is closed under taking subobjects, $S\in \Coh^\beta(Y)_{\mu_{\alpha,\beta}\leq 0}$ and thus $S[1]\in \Coh^0_{\alpha,\beta}(Y)$. Observe that the composition $S[1] \to F'[1] \to F$ is an inclusion in $\Coh^0_{\alpha,\beta}(Y)$, therefore $S[1]$ destabilizes $F$.

Next we prove that $m\leq 1$. It suffices to show that $\Ch_3(F')\leq 1$, since $\Ch (F')=(1,0,-H^2,mH^3)$. By \cite{Li19_FanoPic1}, \cite[Conjecture 4.1]{BMS16} holds for $F'$, for all $(\alpha,\beta)$ where it is semistable. In particular, since $F'$ is semistable along the line $\beta=-\frac{1}{2}$, the inequality holds for $\alpha=0$ and $\beta=-\frac{1}{2}$, which gives
$$4\cdot\frac{49}{64}-6 \frac{1}{2}\Ch_3^\beta(F')\geq 0$$
which simplifies to $\Ch_3 F'\leq3/2.$ This proves $m\leq 1$ (in fact the inequality for $\beta=-1$ gives the exact bound $\Ch_3 F'\leq 1$).
\end{proof}
We now classify all possibilities for $F'$ and $T$ as in Lemma \ref{lem_coh_to_coh0}.

\begin{lemma}\label{lem_stable_k1}
In the setting of Lemma \ref{lem_coh_to_coh0}, 
$F'$ is the ideal sheaf of a one-dimensional subscheme of $Y$. More precisely, there are two possibilities:
\begin{itemize}
    \item[--] if $T=0$, then $F'=I_Z$, for $Z\subset Y$ a subscheme as in Prop. \ref{prop_classification_subschemes};
    \item[--] If $T\neq 0$, then $F'=I_C$, for $C\subset Y$ a genus 1 curve of degree 1 (see Lemma \ref{lem_curves_degree_1}). In this case, $F$ is $F_p$ if $T=\C_p$, and $F$ is one of the $G_x$ if $T=\C_{y_0}$.
\end{itemize}
\end{lemma}
\begin{proof}
Since $\Ch_{\leq 2}(F)=\Ch_{\leq 2}(F')$, Lemma \ref{lem_largest_Wall} shows that there are no walls for $F'$ in the $-1<\beta<0$ strip. Hence $F'$ is $\sigma_{\alpha,\beta}$-semistable for $\alpha\gg 0$. It follows from \cite[Lemma 2.7]{BMS16} that $F'$ is a Gieseker-semistable sheaf.

If $T=0$, then $[F']=\kappa_1$ is one of the ideal sheaves $I_Z$ classified in Prop. \ref{prop_classification_subschemes}. 

Otherwise, $F'$ is an ideal sheaf of a subscheme supported on a curve of degree 1. This is either a line or a genus one curve. It cannot be a line: otherwise, we would have $H^3 =\Ch_3 F'\leq 0$. Hence $F'$ is the ideal sheaf of a genus 1 curve $C$. The only complex with cohomologies $I_C[1]$ and $\C_p$ is $F_p$. Similarly, the $G_x$ are all the complexes with cohomologies $I_C$ and $\C_{y_0}$.
\end{proof}

\begin{proof}[Proof of Theorem \ref{thm_classification}]
The statement about $\sigma_{\alpha,\beta}$-semistable objects is proven with the same argument as \cite[Prop. 4.1]{PY20}: the authors show that $\sigma_{\alpha,\beta}$-stability coincides with Gieseker stability for $\alpha \gg 1$, and that there are no walls for objects of class $-\kappa_1$ on the line $\beta=-\frac 12$.

On the other hand, Propositions \ref{prop_stable_objects1} and \ref{prop_semistable_objects2}, combined with Remark \ref{rmk_IZ1_unstable}, show that $\sigma^0_{\alpha,\beta}$-semistable objects are precisely those listed in the statement.
\end{proof}

\begin{remark}
A simple consequence of Lemma \ref{lem_F_p_stable} is that every $F_p$ is $\sigma(\alpha,\beta)$-stable for all $(\alpha,\beta)\in V$ (defined by Eq. \eqref{eq_def_of_V}). In fact, $F_p$ is actually $\sigma^0_{\alpha,\beta}$-semistable for all $0<\alpha$, $-1<\beta<0$. Since this strip intersects $V$, $F_p$ is also $\sigma(\alpha,\beta)$-semistable, for some $(\alpha,\beta)\in V$, and hence for all of them by Prop. \ref{prop_3.6PY}. Having primitive numerical class, $F_p$ must be $\sigma(\alpha,\beta)$-stable. This proves that the $F_p$ are $\sigma$-stable, giving an alternative argument than that of Theorem \ref{thm_three_moduli}.
\end{remark}

\begin{remark}\label{remark_elem_mod_F_k}
Note that the object $F_{y_0}\in \Coh^0_{\alpha,\beta}(Y)$ is not $\sigma^0_{\alpha,\beta}$-semistable. It is destabilized by the triangle \eqref{eq_destab_tr_Fk}.
However, $F_{y_0}$ is $\sigma(\alpha,\beta)$-stable, since it is the rotation of the stable object $E_{y_0}$ (see Theorem \ref{thm_three_moduli}).
There is no wall in the $(\alpha,\beta)$ plane which would make $F_{y_0}$ stable. Nevertheless, the objects $G_x$ defined in Sec. \ref{sec_obj_and_mod_in_ku} are $\sigma^0_{\alpha,\beta}$-semistable, and they can be obtained from \eqref{eq_destab_tr_Fk} as all the possible extensions in the other direction: in fact, the objects $G$ fitting in a triangle
\begin{equation}
\label{eq_triangle_Gx}
   [\sO_Y^{\oplus 3} \to \sO_Y(1)] \to G\to \sO_Y(-1)[2] 
\end{equation}
are all and only the $G_x$. Indeed, the complex $[\sO_Y^{\oplus 3} \to \sO_Y(1)]$ fits in the Koszul complex
\[ \sO_Y(-2) \xrightarrow{a}\sO_Y(-1)^{\oplus 3} \to \sO_Y^{\oplus 3} \xrightarrow{b} \sO_Y(1) \to \C_{y_0}. \]
Then the cohomology sequence of \eqref{eq_triangle_Gx} gives immediately
\begin{equation}
\label{eq_les_cohomologies_of_G}
\begin{split}
    \sH^0(G)\simeq \C_{y_0},\\
    0\to \sO_Y(-1)  \xrightarrow{c} K \to \sH^{-1}(G),\end{split}
\end{equation}
where $K=\ker(b)=\coker(a)$ and $\sH^{-2}(G)=0$ because $c\neq 0$. Considering the sequence $\sO(-2) \to \sO(-1)^{\oplus 3} \to K$, one sees that $c$ must lift to an inclusion $\sO(-1)\to \sO(-1)^{\oplus 3}$, and hence $\sH^{-1}(G)\simeq \coker(c)=\coker(\sO_Y(-2)\to \sO_Y(-1)^{\oplus 2})=I_{C_x}$ for some $x\in \pr 2$. In other words, $G$ has cohomologies
\[ I_{C_x}[1] \to G \to \C_{y_0} \]
and hence $G\simeq G_x$ for some $x$ (see Remark \ref{rmk_IZ1_unstable}). Conversely, all $G_x$ fit in a triangle \eqref{eq_triangle_Gx}.
\end{remark}

\section{Stable pairs and moduli of \texorpdfstring{$\sigma^0_{\alpha,\beta}$}{sigma-a,b}-semistable complexes}\label{sec_stable_pairs}

In this section we show that there is a fine moduli space for $\sigma^0_{\alpha,\beta}$-semistable complexes. We recall that a \textit{stable pair} on $Y$ is a pair $(P,s)$ where: 
\begin{itemize}
    \item[--] $P$ is a pure sheaf supported on a curve of $Y$;
    \item[--] $s$ is a map
\[\sO_Y \xrightarrow{s} P \]
with zero-dimensional cokernel (see \cite{PT09_CurveCounting}).
\end{itemize}
We say that $\sO_Y\xrightarrow{s} P$ has class $\kappa_1$ if $v(P)=\kappa_1-v(\sO_Y)$.

A \textit{family of stable pairs} over a quasi-projective base scheme $B$ is a pair $(P,s)$ where $P\in \Coh(Y\times B)$ is flat over $B$ and
\[ \sO_{Y\times B} \xrightarrow{s} P, \]
with the property that the restriction $(P_b,s_b)$ is a stable pair on $Y\times\{b\}$ for all closed $b\in B$. 

There is a fine moduli space $P(\kappa_1)$ representing the functor 
\[ \PP(\kappa_1) \colon (\mathrm{Sch/\C})^{\mathrm{op}} \to \mathrm{Sets} \]
whose value on a scheme $B$ is the set of families of stable pairs over $B$ of class $\kappa_1$, and which maps morphisms to pull-backs of families (the space $P(\kappa_1)$ is constructed using GIT techniques and it is projective \cite{LP93}). Pandharipande and Thomas show that two stable pairs $\sO_Y\xrightarrow{s} P$ and $\sO_Y\xrightarrow{s'} P'$ are isomorphic if and only if they are quasi-isomorphic as complexes in $D^b(Y)$ \cite[Prop. 1.21]{PT09_CurveCounting}. As a consequence, they identify $\PP(\kappa_1)$ with the moduli functor whose value in $B$ is the quasi-isomorphism class of $B$-perfect complexes on $Y\times B$ that restrict to stable pairs of class $\kappa_1$ on closed points of $B$ \cite[\S 2]{PT09_CurveCounting}.

On the other hand, consider the weak stability  condition $\sigma_{\alpha,\beta}^0$ of Theorem \ref{thm_classification}. We can define a moduli functor
\begin{align}
     \sM^0_{\alpha,\beta}(\kappa_1) \colon (\mathrm{Sch}/\C)^{\mathrm{op}} \to \mathrm{Gpds}  
\end{align}
whose value on a scheme $B$ is the groupoid of all $B$-perfect complexes $I\in D(Y\times B)$ such that for all closed $b\in B$, $I_b\in D(Y\times \{b\})$ is $\sigma^0_{\alpha,\beta}$-semistable of class $\kappa_1$ (as above, the value of $\sM^0_{\alpha,\beta}(\kappa_1)$ on morphisms is pull-back).

Observe that Theorem \ref{thm_classification} classifies exactly all stable pairs of class $\kappa_1$. In fact, one can argue as in \cite[Lemma 1.6]{PT09_CurveCounting} and show that a stable pair of class $\kappa_1$, viewed as a complex $I\coloneqq [\sO_Y\xrightarrow{s} F]\in D^b(Y)$, satisfies $\sH^0(I)\simeq I_C$ where $C$ is a degree 1 curve, $\mathrm{length}(\sH^1(I))\leq 1$, and all other cohomologies vanish. Such complexes are precisely (shifts of) those in Theorem \ref{thm_classification}.

In other words, $\PP(\kappa_1)$ (interpreted as a moduli functor of complexes) is identified with $\sM^0_{\alpha,\beta}(\kappa_1)$, and therefore $P(\kappa_1)$ is a fine moduli space for $\sM^0_{\alpha,\beta}(\kappa_1)$. 

Moreover (recall the descriptions of $M_G(\kappa_1)$ and $M_\sigma(-\kappa_1)$ in Theorem \ref{thm_Gieseker} and Theorem \ref{thm_three_moduli}) we have:

\begin{thm}\label{thm_stable_pairs}
The projective scheme $P(\kappa_1)$ is a fine moduli space of $\sigma^0_{\alpha,\beta}$-semistable objects, for $\sigma^0_{\alpha,\beta}$ as in Theor. \ref{thm_classification}. $P(\kappa_1)$ contains the surface of lines $F(Y)$, and has a second irreducible component $\tilde{M}_3$, which is the blow-up of $M_3\simeq Y$ at $y_0$.
\end{thm}

\begin{proof}
It follows from Theorem \ref{thm_classification} that the universal family of $M_G(\kappa_1)$, restricted to $F(Y)$, induces an inclusion $F(Y)\to P(\kappa_1)$. We denote by $\tilde{M}_3$ the second irreducible component of $P(\kappa_1)$, it parametrizes complexes of the form $F_p$ for $y_0\neq p\in Y$, and $G_x$ for $x\in \pr 2$. 

Observe first of all that $\tilde{M}_3$ is smooth outside the intersection with the other components, in fact, we have $\mathrm{ext}^1(F_p,F_p)=3$ (by Theorem \ref{thm_three_moduli}) and $\mathrm{ext}^1(G_x,G_x)=3$ by Lemma \ref{lem_ext1GG} below. Then, the locus $D'$ parametrizing the objects $G_x$ lies in the smooth locus of $\tilde{M}_3$, and hence $D'$ is a Cartier divisor in $\tilde{M}_3$. Set $D=D'\times Y$ and write $i_D\colon D \to \tilde{M}_3\times Y$ for the inclusion.

Let $\sI\in D^b(\tilde{M}_3\times Y)$ be the universal  family of $P(\kappa_1)$ restricted to $\tilde{M}_3$. We will use a modification of $\sI$ to construct a family of objects of $\Ku(Y)$ supported on $\tilde{M}_3$. Consider the triangle 
\[ \sI(-D) \to \sI \xrightarrow{r} \sI_{|D}  \]
and the relative version of \eqref{eq_triangle_Gx} over the projection $p_D\colon D \to Y$:
\begin{equation*}\label{eq_triangle_def_of_u}
    p_D^*[\sO_Y^{\oplus 3}\to \sO_Y(1)]  \to \sI_{|D} \xrightarrow{u} p_D^*\sO_Y(-1)[2]
\end{equation*}
(we denote $A\coloneqq [\sO_Y^{\oplus 3}\to \sO_Y(1)] $ in what follows). We abuse notation and we use the same letter $u$ for the map $i_{D*}\sI_{|D} \xrightarrow{u} i_{D*}p_D^*\sO_Y(-1)[2]$ obtained by pushing forward.
The octahedral axiom applied to $u\circ r$ yields a triangle
\begin{equation} \label{eq_octahedral_outcome}
   \sI(-D) \to \sI' \to i_{D*}p_D^*A, 
\end{equation}
where $\sI'$ is the (shift of the) cone of $u\circ r$. By tensoring $i_{D*}p_D^*A$ with the sequence $\sO_{\tilde{M}_3\times Y}(-D)\to \sO_{\tilde{M}_3\times Y} \to i_{D*}\sO_D$ we obtain a triangle on $D$:
\begin{equation}
    \label{eq_triang_derived_restriction}
p_D^*A(-D) \xrightarrow{0} p_D^*A \to \mathbb{L}i^*_D i_{D*}p_D^*A,  
\end{equation}
where $\mathbb{L}i^*_D i_{D*}p_D^*A$ is the derived restriction of $i_{D*}p_D^*A$ to $D$.

On the other hand, restriction of \eqref{eq_octahedral_outcome} to a fiber $[G_x]\times Y$ of $D$, gives a triangle

\begin{equation}
    \label{eq_octahedral_outcome_on_fibers}
 G_x \to  (\mathbb{L}i_D^*\sI')_{[G_x]} \to (\mathbb{L}i^*_D i_{D*}p_D^*A)_{[G_x]}. \end{equation}
 
The cohomologies of $(\mathbb{L}i^*_D i_{D*}p_D^*A)_{[G_x]}$ can be computed from \eqref{eq_triang_derived_restriction}, using that the cohomologies of $A$ (and those of $G_x$) are computed in Remark \ref{remark_elem_mod_F_k}. Then, taking cohomologies of \eqref{eq_octahedral_outcome_on_fibers}, we see
\begin{center}
\def\arraystretch{1.3}
\begin{tabular}{ccccc}
 $G_x $ &$\to$ &  $(\mathbb{L}i_D^*\sI')_{[G_x]} $ &$\to$ & $(\mathbb{L}i^*_D i_{D*}p_D^*A)_{[G_x]}$\\\hline
$0 $ &$ \to$ & $\sH^{-2}((\mathbb{L}i_D^*\sI')_{[G_x]}) $ &$\to$ & $K \to$\\
$I_{C_x}$ &$\xrightarrow{f} $ & $\sH^{-1}((\mathbb{L}i_D^*\sI')_{[G_x]}) $ &$ \to $ & $M \to$\\
$\C_{y_0} $ &$\xrightarrow{g}$ & $ \sH^{0}((\mathbb{L}i_D^*\sI')_{[G_x]}) $ &$ \to  $ & $\C_{y_0}$\\
\hline
\end{tabular}
\end{center}
where $K=h^{-1}(A)$ and $M$ is an extension of $\C_{y_0}$ by $K$. By construction, the connecting map $K\to I_{C_x}$ coincides with that in \eqref{eq_les_cohomologies_of_G}. So $\sH^{-2}((\mathbb{L}i_D^*\sI')_{[G_x]}) \simeq \sO_Y(-1)$ and $f=0$. Similarly, the map $M\to \C_{y_0}$ is surjective, so that $\sH^{-1}((\mathbb{L}i_D^*\sI')_{[G_x]}) \simeq K$ and $g=0$, which implies $\sH^{0}((\mathbb{L}i_D^*\sI')_{[G_x]}) \simeq \C_{y_0}$. 

Therefore there is a triangle 
\[ \sO_Y(-1)[2] \to (\mathbb{L}i_D^*\sI')_{[G_x]} \to A, \]
which shows that $(\mathbb{L}i_D^*\sI')_{[G_x]}\simeq F_{y_0}$ by Lemma \ref{lem_rot_at_k} (and because $\mathrm{ext}^1(A,\sO_Y(-1)[2])=1$).

Then, $\sI'\in D(\tilde{M}_3\times Y)$ defines a flat family of $\sigma$-stable objects of $\Ku(Y)$ of class $\kappa_1$, and therefore a morphism $\tilde{M}_3\to M_3$ which maps $D'$ to $y_0$. \qedhere

\end{proof}

\begin{lemma}\label{lem_ext1GG}
For $x\in \pr 2$, we have $\mathrm{ext}^1(G_x,G_x)=3$. 
\end{lemma}

\begin{proof}
We compute this applying the spectral sequence \eqref{eq_spec_seq} to the complex $G_x\simeq [\sO_Y \to \sO_{C_x}(y_0)]$. The first page has spaces of dimension
\begin{center}
\begin{tabular}{c|cc}
0 & 0 &0\\
1 & * &0\\
0 & $0+2$ &0\\
0 & $1+1$ &1\\
\hline
\end{tabular}
\end{center}
Since the map in the bottom row  
\[ \Hom(\sO_Y,\sO_Y)\oplus \Hom(\sO_{C_x}(y_0),\sO_{C_x}(y_0)) \to \Hom(\sO_Y,\sO_{C_x}(y_0))\]
is non-zero, $\mathrm{ext}^1(G_x,G_x)\leq 3$. But every $G_x$ fits in a three dimensional component (there are two dimensions for deforming $C_x$ and one to move $y_0$), so $\mathrm{ext}^1(G_x,G_x)\geq 3$ and equality holds.
\end{proof}

\subsection{\texorpdfstring{$P(\kappa_1)$}{P(k1)} as a generalized Quot scheme}

In Section \ref{sec_Fano_scheme_of_lines} we showed that $M_G(\kappa_1)$ is isomorphic to the Hilbert scheme of lines on $Y$. Here, we give a similar interpretation for the moduli space $P(\kappa_1)$ of $\sigma^0_{\alpha,\beta}$-semistable objects as quotients of $\sO_Y$ in an appropriate heart of $D^b(Y)$.

Consider the sheaves of the form $\sO_C(p)$, where $p\in Y$ (possibly $p=y_0$) and $C=C_x$ for some $x\in \pr 2$. By  Riemann-Roch we also have
\[ \chi(\sO_C(p)(t))=1+t. \]
However, the $\sO_C(p)$ are not sheaf quotients of $\sO_Y$ and do not represent points of the Hilbert scheme of lines of $Y$. 

In this section, we consider a different space of quotients, and show that the distinguished triangles
\begin{equation}
\label{eq_ses_quot}
    \begin{split}
         I_\ell \to &\;\sO_Y \to \sO_\ell\\
         F_p[-1] \to &\;\sO_Y \to \sO_C(p)\\
        G_x[-1] \to &\;\sO_Y \to \sO_{C_x}(y_0)
    \end{split}
\end{equation}
are all short exact sequences in an appropriate abelian category (the notation here is the same as that of Theorem \ref{thm_classification}). 

More precisely, define $\mathcal{B}^\theta_{\alpha,\beta}$ as follows: pick $(\alpha,\beta)\in V$ so that the chain of inequalities 
\[ \frac{\beta}{-\left(\frac{\alpha^2-\beta^2}{2}+1\right)}=
\mu^0_{\alpha,\beta}(F_p)>0=\mu^0_{\alpha,\beta}(\sO_C(p)) > 
\mu^0_{\alpha,\beta}(\sO_Y)=\frac{2\beta}{-(\alpha^2-\beta^2)} \]
is satisfied. By the wall computation of Lemma \ref{lem_largest_Wall}, we may pick $0<\epsilon\ll 1$ so that, if $F$ is any unstable object of class $-\kappa_1$, then any destabilizing quotient $G$ satisfies $\mu^0_{\alpha,\beta}(G)< \theta\coloneqq \mu^0_{\alpha,\beta}(F)-\epsilon$. 

Then, consider the torsion pair in $\Coh^0_{\alpha,\beta}(Y)$ consisting of the categories $\Coh^0_{\alpha,\beta}(Y)_{\mu^0_{\alpha,\beta}\leq \theta}$ and $\Coh^0_{\alpha,\beta}(Y)_{\mu^0_{\alpha,\beta} > \theta}$ generated by $\sigma^0_{\alpha,\beta}$-semistable objects of slope $\leq \theta$ and $>\theta$ respectively. Denote by $\sB^\theta_{\alpha,\beta}$ the (shift of the) corresponding tilt:
\[\sB^\theta_{\alpha,\beta} \coloneqq \left[ \Coh^0_{\alpha,\beta}(Y)_{\mu^0_{\alpha,\beta}\leq \theta}, \Coh^0_{\alpha,\beta}(Y)_{\mu^0_{\alpha,\beta} > \theta}[-1] \right].\]

Since $I_Z[1]$, $F_p$, and $G_x$ are $\sigma^0_{\alpha,\beta}$-semistable of phase $>\theta$, their shifts by $-1$ belong to $\mathcal{B}^\theta_{\alpha,\beta}$. Similarly, by the choice of $\theta$ we have that $\sO_\ell$, $\sO_C(p)$, $\sO_{C_x}(y_0)$, and $\sO_Y$ belong to $ \mathcal{B}^\theta_{\alpha,\beta}$ as well. Then, the triangles in \eqref{eq_ses_quot} are short exact sequences in $\mathcal{B}^\theta_{\alpha,\beta}$. The converse is true:

\begin{proposition}\label{prop_quot}
Quotients of $\sO_Y$ of class $\kappa_1$ in $\mathcal{B}^\theta_{\alpha,\beta}$ are precisely the objects $\sO_\ell$, $\sO_C(p)$, and $\sO_{C_x}(y_0)$ listed in \eqref{eq_ses_quot}.
\end{proposition}

\begin{proof}
First, we claim that if 
\begin{equation}
    \label{eq_ex_sec_quot}
    F \to \sO_Y \to Q 
\end{equation}
is a short exact sequence in $\sB^\theta_{\alpha,\beta}$ with $\chi(Q(t))=t+1$, then $F$ is $\sigma^0_{\alpha,\beta}$-semistable of class $\kappa_1$. Indeed, the statement about the numerical class is immediate. As for semistability: a destabilizing quotient $G$ of $F$ must have $\mu^0_{\alpha,\beta}(F) > \mu^0_{\alpha,\beta}(G) >  \theta$,  otherwise $G\notin \sB^\theta_{\alpha,\beta}$. But this contradicts our choice of $\theta$.

So, $F$ must be (a shift of) the objects classified in Theorem \ref{thm_classification}, and the sequence \eqref{eq_ex_sec_quot} must be one of those listed in \eqref{eq_ses_quot}.
\end{proof}

\begin{remark}
The arguments above identify the moduli functor $\sM^0_{\alpha,\beta}(\kappa_1)$ with the generalized Quot functor defined in \cite[Sec. 11]{BLMNPS19} and \cite{Rota2019}.
\end{remark}

\bibliographystyle{amsalpha}
\bibliography{bibliography_APR.bib}

\newcommand{\etalchar}[1]{$^{#1}$}
\providecommand{\bysame}{\leavevmode\hbox to3em{\hrulefill}\thinspace}
\providecommand{\MR}{\relax\ifhmode\unskip\space\fi MR }
\providecommand{\MRhref}[2]{%
  \href{http://www.ams.org/mathscinet-getitem?mr=#1}{#2}
}
\providecommand{\href}[2]{#2}
\begin{thebibliography}{BLM{\etalchar{+}}21}

\bibitem[APR22]{APR19}
Matteo Altavilla, Marin Petkovic, and Franco Rota, \emph{{Moduli spaces on the
  Kuznetsov component of Fano threefolds of index 2}}, {Épijournal de
  Géométrie Algébrique} \textbf{{Volume} 6} (2022).

\bibitem[BLM{\etalchar{+}}21]{BLMNPS19}
Arend Bayer, Mart\'{\i} Lahoz, Emanuele Macr\`\i, Howard Nuer, Alexander Perry,
  and Paolo Stellari, \emph{Stability conditions in families}, Publ. Math.
  Inst. Hautes \'{E}tudes Sci. \textbf{133} (2021), 157--325. \MR{4292740}

\bibitem[BLMS17]{BLMS17}
Arend {Bayer}, Mart{\'\i} {Lahoz}, Emanuele {Macr{\`\i}}, and Paolo {Stellari},
  \emph{{Stability conditions on Kuznetsov components}}, arXiv e-prints (2017),
  arXiv:1703.10839.

\bibitem[BMS16]{BMS16}
Arend Bayer, Emanuele Macr\`\i, and Paolo Stellari, \emph{The space of
  stability conditions on abelian threefolds, and on some {C}alabi-{Y}au
  threefolds}, Invent. Math. \textbf{206} (2016), no.~3, 869--933. \MR{3573975}

\bibitem[BMT14]{BMT14}
Arend Bayer, Emanuele Macr\`\i, and Yukinobu Toda, \emph{Bridgeland stability
  conditions on threefolds {I}: {B}ogomolov-{G}ieseker type inequalities}, J.
  Algebraic Geom. \textbf{23} (2014), no.~1, 117--163. \MR{3121850}

\bibitem[Bri07]{Bri07_triang_cat}
Tom Bridgeland, \emph{Stability conditions on triangulated categories}, Ann. of
  Math. (2) \textbf{166} (2007), no.~2, 317--345.

\bibitem[CK11]{CK11}
Kiryong Chung and Young-Hoon Kiem, \emph{Hilbert scheme of rational cubic
  curves via stable maps}, Amer. J. Math. \textbf{133} (2011), no.~3, 797--834.
  \MR{2808332}

\bibitem[FP21]{FP21}
Soheyla {Feyzbakhsh} and Laura {Pertusi}, \emph{{Serre-invariant stability
  conditions and Ulrich bundles on cubic threefolds}}, arXiv e-prints (2021),
  arXiv:2109.13549.

\bibitem[HK15]{HK15}
Jun-Muk Hwang and Hosung Kim, \emph{Varieties of minimal rational tangents on
  {V}eronese double cones}, Algebr. Geom. \textbf{2} (2015), no.~2, 176--192.
  \MR{3350155}

\bibitem[Isk77]{Iskovskih1977}
Vasily~A. Iskovskih, \emph{Fano threefolds. {I}}, Izv. Akad. Nauk SSSR Ser.
  Mat. \textbf{41} (1977), no.~3, 516--562, 717. \MR{463151}

\bibitem[KPS18]{Kuznetsov2018}
Alexander~G. Kuznetsov, Yuri~G. Prokhorov, and Constantin~A. Shramov,
  \emph{Hilbert schemes of lines and conics and automorphism groups of {F}ano
  threefolds}, Jpn. J. Math. \textbf{13} (2018), no.~1, 109--185. \MR{3776469}

\bibitem[Kuz09]{Kuz09_threefolds}
Alexander~G. Kuznetsov, \emph{Derived categories of {F}ano threefolds}, Tr.
  Mat. Inst. Steklova \textbf{264} (2009), no.~Mnogomernaya Algebraicheskaya
  Geometriya, 116--128. \MR{2590842}

\bibitem[Kuz14]{Kuz14}
\bysame, \emph{Semiorthogonal decompositions in algebraic geometry},
  Proceedings of the {I}nternational {C}ongress of {M}athematicians---{S}eoul
  2014. {V}ol. {II}, Kyung Moon Sa, Seoul, 2014, pp.~635--660.

\bibitem[Kuz19]{Kuz15_CY}
Alexander Kuznetsov, \emph{Calabi-{Y}au and fractional {C}alabi-{Y}au
  categories}, J. Reine Angew. Math. \textbf{753} (2019), 239--267.
  \MR{3987870}

\bibitem[Li19]{Li19_FanoPic1}
Chunyi Li, \emph{Stability conditions on {F}ano threefolds of {P}icard number
  1}, J. Eur. Math. Soc. (JEMS) \textbf{21} (2019), no.~3, 709--726.
  \MR{3908763}

\bibitem[LP93]{LP93}
Joseph Le~Potier, \emph{Syst\`emes coh\'{e}rents et structures de niveau},
  Ast\'{e}risque (1993), no.~214, 143. \MR{1244404}

\bibitem[LZ21]{LZ21}
Zhiyu {Liu} and Shizhuo {Zhang}, \emph{{A note on Bridgeland moduli spaces and
  moduli spaces of sheaves on $X_{14}$ and $Y_3$}}, arXiv e-prints (2021),
  arXiv:2106.01961.

\bibitem[Mac14]{Maciocia14}
Antony Maciocia, \emph{Computing the walls associated to {B}ridgeland stability
  conditions on projective surfaces}, Asian J. Math. \textbf{18} (2014), no.~2,
  263--279. \MR{3217637}

\bibitem[MS17]{MS17}
Emanuele Macr\`i and Benjamin Schmidt, \emph{Lectures on {B}ridgeland
  stability}, Moduli of curves, Lect. Notes Unione Mat. Ital., vol.~21,
  Springer, Cham, 2017, pp.~139--211. \MR{3729077}

\bibitem[PS85]{PS85}
Ragni Piene and Michael Schlessinger, \emph{On the {H}ilbert scheme
  compactification of the space of twisted cubics}, Amer. J. Math. \textbf{107}
  (1985), no.~4, 761--774. \MR{796901}

\bibitem[PS22]{PS22}
Laura {Pertusi} and Paolo {Stellari}, \emph{{Categorical Torelli theorems:
  results and open problems}}, arXiv e-prints (2022), arXiv:2201.03899.

\bibitem[PT09]{PT09_CurveCounting}
Rahul Pandharipande and Richard~P. Thomas, \emph{Curve counting via stable
  pairs in the derived category}, Invent. Math. \textbf{178} (2009), no.~2,
  407--447. \MR{2545686}

\bibitem[PY21]{PY20}
Laura Pertusi and Song Yang, \emph{{Some Remarks on Fano Three-Folds of Index
  Two and Stability Conditions}}, International Mathematics Research Notices
  (2021), https://doi.org/10.1093/imrn/rnaa387.

\bibitem[{Qin}21]{Qin21}
Xuqiang {Qin}, \emph{{Bridgeland stability of minimal instanton bundles on Fano
  threefolds}}, arXiv e-prints (2021), arXiv:2105.14617.

\bibitem[Rot19]{Rota2019}
Franco Rota, \emph{Moduli spaces of sheaves: generalized {Q}uot schemes and
  {B}ridgeland stability conditions}, Ph.D. thesis, University of Utah, 2019.

\bibitem[Sch20]{Sch20}
Benjamin Schmidt, \emph{Bridgeland stability on threefolds: some wall
  crossings}, J. Algebraic Geom. \textbf{29} (2020), no.~2, 247--283.
  \MR{4069650}

\bibitem[Tih82]{Tihomirov1982}
Alexander~S. Tihomirov, \emph{The {F}ano surface of the {V}eronese double
  cone}, Mathematics of the {USSR}-Izvestiya \textbf{19} (1982), no.~2,
  377--443.

\end{thebibliography}
\end{document}